\documentclass[11pt,a4paper]{article}

\usepackage{titlesec}
\usepackage{fancyhdr}
\usepackage{a4wide}
\usepackage{graphicx}
\usepackage{float}
\usepackage{amssymb}
\usepackage{amsmath}
\usepackage{amsfonts}
\usepackage{amsthm}
\usepackage{color}
\usepackage{mathrsfs}
\usepackage{array}
\usepackage{eucal}
\usepackage{tikz}
\usepackage[T1]{fontenc}
\usepackage{inputenc}
\usepackage[english]{babel}
\usepackage{lmodern}
\usepackage{hyperref}
\usepackage{geometry}
\usepackage{changepage}
\usepackage{bm}
\changepage{0pt}{}{}{}{}{0pt}{}{0pt}{6pt}
\usepackage[numbers]{natbib}
\usepackage{hyperref}
\hypersetup{
pdfpagemode=none,
pdftoolbar=true,        
pdfmenubar=true,        
pdffitwindow=false,     
pdfstartview={Fit},    
pdftitle={Design of a mode converter using thin resonant ligaments},    
pdfauthor={L. Chesnel, J. Heleine, S.A. Nazarov},     
pdfsubject={},  
pdfcreator={L. Chesnel, J. Heleine, S.A. Nazarov},   
pdfproducer={L. Chesnel, J. Heleine, S.A. Nazarov}, 
pdfkeywords={}, 
pdfnewwindow=true,      
colorlinks=true,       
linkcolor=magenta,          
citecolor=red,        
filecolor=cyan,      
urlcolor=blue           
}

\newcommand{\dsp}{\displaystyle}

\newcommand{\eps}{\varepsilon}
\newcommand{\om}{\omega}
\newcommand{\Om}{\Omega}
\newcommand{\mrm}[1]{\mathrm{#1}}

\newcommand{\Cplx}{\mathbb{C}}
\newcommand{\N}{\mathbb{N}}
\newcommand{\R}{\mathbb{R}}

\newtheorem{theorem}{Theorem}[section]
\newtheorem{lemma}[theorem]{Lemma}
\newtheorem{remark}[theorem]{Remark}

\newtheorem{proposition}[theorem]{Proposition}

\begin{document}

~\vspace{0.0cm}
\begin{center}
{\sc \bf\huge Design of a mode converter\\[6pt] using thin resonant ligaments}
\end{center}

\begin{center}
\textsc{Lucas Chesnel}$^1$, \textsc{J\'er\'emy Heleine}$^1$, \textsc{Sergei A. Nazarov}$^{2}$\\[16pt]
\begin{minipage}{0.95\textwidth}
{\small
$^1$ INRIA/Centre de math\'ematiques appliqu\'ees, \'Ecole Polytechnique, Institut Polytechnique de Paris, Route de Saclay, 91128 Palaiseau, France;\\
$^2$ St. Petersburg State University, Universitetskaya naberezhnaya, 7-9, 199034, St. Petersburg, Russia;\\[10pt]
E-mails: \texttt{lucas.chesnel@inria.fr}, \texttt{jeremy.heleine@inria.fr}, \texttt{srgnazarov@yahoo.co.uk} \\[-14pt]
\begin{center}
(\today)
\end{center}
}
\end{minipage}
\end{center}
\vspace{0.4cm}

\noindent\textbf{Abstract.} 
The goal of this work is to design an acoustic mode converter. More precisely, the wave number is chosen so that two modes can propagate. We explain how to construct geometries such that the energy of the modes is completely transmitted and additionally the mode 1 is converted into the mode 2 and conversely. To proceed, we work in a symmetric waveguide made of two branches connected by two thin ligaments whose lengths and positions are carefully tuned. The approach is based on asymptotic analysis for thin ligaments around resonance lengths. We also provide numerical results to illustrate the theory.\\


\section{Introduction}\label{Introduction}

The mode converter is a classical device in waves physics which appears for example in optics \cite{HoHu05,CXJD06,KTSM09,LiMF12,OhLe14} or in the field of microwaves \cite{Micro2,Micro1}. In this article, we consider a rather academic but universal problem of propagation of acoustic waves in a waveguide which is bounded in the transverse direction. This problem also arises in electromagnetism and in water-wave theory in certain configurations. We work in time-harmonic regime and the wavenumber is set so that two modes, say modes 1 and 2, can propagate in the structure. When the mode 1 or the mode 2 is incoming in the structure, in general it produces a reflection on the modes 1, 2 and a transmission on the modes 1, 2. The latter are characterized by some complex reflection and transmission coefficients (see (\ref{Field1}), (\ref{Field2}), (\ref{DefScaMat}) for precise definitions).
The goal of this work is to construct a waveguide, that is a geometry, which exhibits two main features at the given wavenumber. First, we want the reflection coefficients to be null so that the energy of any incident field is completely transmitted. Secondly, we wish the energy of the incident mode 1 to be transmitted only on the mode 2 and vice versa to guarantee the mode conversion.\\
\newline
The main difficulty of the problem lies in the fact that the dependence of the scattering coefficients with respect to the geometry is not explicit and not linear. In order to address it, techniques of optimization have been considered. We refer the reader in particular to \cite{KuTW88,LuKG98,LAHI00,LDOHG19,LGHDO19,Lebb19}. However the functionals involved in the analysis are non convex and unsatisfying local minima exist. Moreover, these methods do not allow the user to control the main features of the shape compare to the approach we propose below.\\
\newline
To construct mode converters, we will work in waveguides as in Figure \ref{DomainOriginal} made of two truncated channels connected by thin ligaments of width $\eps>0$. Let us mention that we used a similar approach in \cite{ChNa20Distrib} to design energy distributors, that is geometries with three channels where in monomode regime, the energy of an incoming wave is almost completely transmitted and where additionally one can control the ratio of energy transmitted in the two other channels. The present article reuses some analysis presented in \cite{ChNa20Distrib} and can be seen as a second part of \cite{ChNa20Distrib}. In general, due to the geometrical properties of waveguides as in Figure \ref{DomainOriginal}, almost no energy passes through the ligaments and it does not seem relevant to exploit them to get almost complete transmission. However working around the resonance lengths of the ligaments (see (\ref{ConditionResonance})), it has been shown that this can be achieved. This has been studied for example in \cite{Krieg,BoTr10,LiZh17,LiZh18,LiSZ19} in the context of the scattering of an incident wave by a periodic array of subwavelength slits. The core of our approach is based on an asymptotic expansion of the scattering solutions with respect to $\eps$ as $\eps$ tends to zero. This will allow us to derive relatively explicit formula relating the scattering properties to the geometrical features. To obtain the expansions, we will apply techniques of matched asymptotic expansions. For related methods, we refer the reader to \cite{Beal73,Gady93,KoMM94,Naza96,Gady05,Naza05,JoTo06,BaNa15,BoCN18}. We emphasize that an important feature of our study distinguishing it from the previous references is that the lengths, and not only the widths, of the ligaments depend on $\eps$ (see (\ref{DefL})). This way of considering the problem, which we used also in \cite{NaCh21a,NaCh21b,ChNa20Distrib}, is an essential ingredient to reveal the resonant phenomena. From this perspective, our work shares similarities with \cite{HoSc19,BrHS20,BrSc20} (see also references therein).\\
\newline
The outline is as follows. In the next section, we describe the setting. We work in waveguides which are symmetric with respect to the vertical axis and we explain how to split the problem into two half-waveguide problems with Neumann or Dirichlet boundary condition at the end of the truncated ligaments. In Section \ref{SectionAux}, we introduce auxiliary objects that will be used in the asymptotic analysis. Sections \ref{SectionAsymptoNeumann} and \ref{SectionAsymptoDirichlet} form the principal part of the article. There we compute asymptotic expansions of the scattering solutions for the problems with respectively Neumann and Dirichlet boundary condition at the end of the truncated ligaments as $\eps$ tends to zero. Then we exploit the results in Section \ref{AnalysisResults} to provide examples of geometries acting as mode converters. Then we illustrate the theory in Section \ref{SectionNumerics} with numerical experiments before discussing possible extensions and open questions in Section \ref{sectionConclu}. Finally, we give the proof of two technical lemmas needed in the study in a short appendix. The main result of this article appears in Proposition \ref{MainPropo}.

\section{Setting}

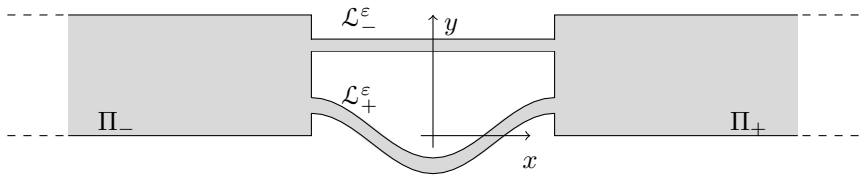
\begin{figure}[!ht]
\centering
\begin{tikzpicture}[scale=1.6]
\draw[fill=gray!30](-0.2,0.7) rectangle (2,0.8);
\draw[fill=gray!30,draw=none](-2,0) rectangle (0,1);
\draw[fill=gray!30,draw=none](2,0) rectangle (4,1);
\draw (-2,0)--(0,0)--(0,1)--(-2,1);
\draw (4,0)--(2,0)--(2,1)--(4,1);
\draw[dashed] (-2.5,0)--(-2,0);
\draw[dashed] (-2.5,1)--(-2,1);
\draw[dashed] (4.5,0)--(4,0);
\draw[dashed] (4.5,1)--(4,1);
\draw[fill=gray!30,draw=none](-0.2,0.705) rectangle (2.1,0.795);
\draw[domain=-0.1:2.1,line width=2mm,gray!30]plot(\x, { 0.25*cos(\x*180)}  );
\draw[domain=0:2]plot(\x, { 0.065+0.25*cos(\x*180)}  );
\draw[domain=0:2]plot(\x, { -0.065+0.25*cos(\x*180)}  );
\begin{scope}[shift={(0.9,-0.2)}]
\draw[->] (0,0.2)--(0.9,0.2);
\draw[->] (0.1,0.1)--(0.1,1.2);
\node at (0.9,-0.02){\small $x$};
\node at (0.25,1.1){\small $y$};
\end{scope}
\node at (-1.6,0.1){\small $\Pi_-$};
\node at (3.6,0.1){\small $\Pi_+$};
\node at (0.4,0.95){\small $\mathcal{L}^{\eps}_-$};
\node at (0.4,0.3){\small $\mathcal{L}^{\eps}_+$};
\end{tikzpicture}
\caption{Geometry of the waveguide $\Om^\eps$ with two thin ligaments. \label{DomainOriginal}} 
\end{figure}

Define the domains $\Pi_{\pm}:=\{z=(x,y)\in\R^2\,|\,(\pm x,y)\in(1/2;+\infty)\times(0;1)\}$. Let us connect them by two thin non intersecting ligaments $\mathcal{L}^{\eps}_{-}$, $\mathcal{L}^{\eps}_{+}$ of constant width $\eps>0$ whose features will be made more precise below. Define the waveguide $\Om^{\eps}$ 
\begin{equation}\label{DomainFirstDef}
\Om^{\eps}=\Pi_-\cup \mathcal{L}^{\eps}_{-} \cup \mathcal{L}^{\eps}_{+} \cup \Pi_+
\end{equation}
(see Figure \ref{DomainOriginal}). We assume that $\Om^{\eps}$ is connected and that its boundary $\partial\Om^{\eps}$ is Lipschitz. Interpreting the domain $\Om^{\eps}$ as an acoustic waveguide, we are led to consider the following problem with Neumann Boundary Conditions (BC)
\begin{equation}\label{MainPb}
 \begin{array}{|rcll}
 \Delta u^\eps +\omega^2  u^\eps&=&0&\mbox{ in }\Omega^\eps\\
 \partial_\nu u^\eps &=&0 &\mbox{ on }\partial\Omega^\eps.
\end{array}
\end{equation}
Here, $\Delta$ is the Laplace operator while $\partial_\nu$ corresponds to the derivative along the exterior normal. Furthermore, $u^\eps$ is the acoustic pressure of the medium while $\omega>0$ is the wave number. We fix $\om\in(\pi;2\pi)$ so that only the two  modes
\[
\mrm{w}^\pm_1(x,y)=\cfrac{1}{\sqrt{\beta_1}}\,e^{\pm i\beta_1x}\qquad\mbox{ and }\qquad \mrm{w}^\pm_2(x,y)=\cfrac{1}{\sqrt{\beta_2}}\,e^{\pm i\beta_2x}\sqrt{2}\cos(\pi y)
\]
with $\beta_1:=\om$, $\beta_2:=\sqrt{\om^2-\pi^2}$ can propagate. Note that the normalization factors are chosen to have some simple relations of conservation of energy (see (\ref{RelConsNRJ}) and (\ref{RelConsNRJDemi})). We are interested in the solutions to the diffraction problem \eqref{MainPb} generated by the incoming waves $\mrm{w}^+_1$, $\mrm{w}^+_2$ in the channel $\Pi_-$. These solutions admit the decompositions
\begin{equation}\label{Field1}
 u^\eps_1(x,y)=\begin{array}{|ll}
\mrm{w}^+_1(x+1/2,y)+r^{\eps}_{11}\mrm{w}^-_1(x+1/2,y)+r^{\eps}_{12}\mrm{w}^-_2(x+1/2,y)+\dots &\quad\mbox{ in }\Pi_-\\[3pt]
\phantom{\mrm{w}^+_1(x+1/2,y)+\ \,} t^{\eps}_{11}\mrm{w}^+_1(x-1/2,y)+t^{\eps}_{12}\mrm{w}^+_2(x-1/2,y)+\dots &\quad\mbox{ in }\Pi_+\end{array}
\end{equation}
\begin{equation}\label{Field2}
 u^\eps_2(x,y)=\begin{array}{|ll}
\mrm{w}^+_2(x+1/2,y)+r^{\eps}_{21}\mrm{w}^-_1(x+1/2,y)+r^{\eps}_{22}\mrm{w}^-_2(x+1/2,y)+\dots &\quad\mbox{ in }\Pi_-\\[3pt]
\phantom{\mrm{w}^+_2(x+1/2,y)+\ \,} t^{\eps}_{21}\mrm{w}^+_1(x-1/2,y)+t^{\eps}_{22}\mrm{w}^+_2(x-1/2,y)+\dots &\quad\mbox{ in }\Pi_+\end{array}
\end{equation}
where the $r^{\eps}_{ij}\in\mathbb{C}$ are reflection coefficients and $t^{\eps}_{ij}\in\mathbb{C}$ are transmission coefficients. Here the ellipsis stand for remainders which decay at infinity with the rate $e^{-(4\pi^2-\om^2)^{1/2}|x|}$ in $\Pi_\pm$. Due to conservation of energy, for $i=1,2$, one can verify that there holds
\begin{equation}\label{RelConsNRJ}
|r^{\eps}_{i1}|^2+|r^{\eps}_{i2}|^2+|t^{\eps}_{i1}|^2+|t^{\eps}_{i2}|^2=1.
\end{equation}
We define the reflection and transmission matrices
\begin{equation}\label{DefScaMat}
R^{\eps}:=\left(\begin{array}{cc}
r^{\eps}_{11} & r^{\eps}_{12} \\[6pt]
r^{\eps}_{21} & r^{\eps}_{22} 
\end{array}\right)\qquad\mbox{ and }\qquad T^{\eps}:=\left(\begin{array}{cc}
t^{\eps}_{11} & t^{\eps}_{12} \\[6pt]
t^{\eps}_{21} & t^{\eps}_{22} 
\end{array}\right).
\end{equation}
In general, due to the geometrical features of waveguides such as the one depicted in Figure \ref{DomainOriginal}, almost no energy of the incident waves $\mrm{w}^+_1$, $\mrm{w}^+_2$ passes through the thin ligaments $\mathcal{L}^{\eps}_{\pm}$ and one observes almost complete reflection. More precisely, one finds that as $\eps$ tends to zero, there holds 
\begin{equation}\label{DefExpanNonCri}
 R^\eps  =\left(\begin{array}{cc}
1 & 0 \\[1pt]
0 & 1 
\end{array}\right)+o(1),\qquad\qquad T^\eps  =\left(\begin{array}{cc}
0 & 0 \\[1pt]
0 & 0 
\end{array}\right)+o(1)
\end{equation}
(see the numerics in Figure \ref{FieldNotTuned}). The main goal of this work is to show that by choosing carefully the lengths of the thin ligaments as well as their positions, the energy of the waves $\mrm{w}^+_1$, $\mrm{w}^+_2$ can be almost completely transmitted. Moreover we can obtain mode conversion, that is we can ensure that all the energy carried by the incident mode 1 be transferred on the mode 2 and vice versa. In terms of scattering matrices, we will establish that by choosing carefully the properties of the ligaments, as $\eps$ tends to zero, we can have
\begin{equation}\label{ConditionModalConverter}
 R^\eps  =\left(\begin{array}{cc}
0 & 0 \\[1pt]
0 & 0 
\end{array}\right)+o(1),\qquad\qquad T^\eps  =\left(\begin{array}{cc}
0 & 1 \\[1pt]
1 & 0 
\end{array}\right)+o(1).
\end{equation}
To diminish the number of parameters to play with, we work with waveguides which are symmetric with respect to the $(Oy)$ axis. In other words, we assume that $\Om^{\eps}=\{(x,y)\,|(-x,y)\in\Om^{\eps}\}$ and we set $\om^{\eps}:=\{(x,y)\,\in\Om^{\eps}\,|\,x<0\}$ (see Figure \ref{DomainOriginalHalf}). The solutions of the initial problem (\ref{MainPb}) in $\Om^{\eps}$ can be expressed by means of solutions of two problems set in $\om^{\eps}$ that we present now.

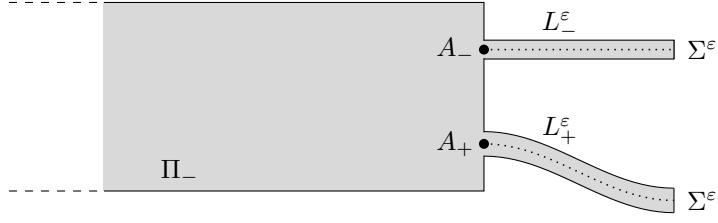
\begin{figure}[!ht]
\centering
\begin{tikzpicture}[scale=2.5]
\draw[fill=gray!30](-0.2,0.7) rectangle (1,0.8);
\draw[fill=gray!30,draw=none](-2,0) rectangle (0,1);
\draw (-2,0)--(0,0)--(0,1)--(-2,1);
\draw[dashed] (-2.5,0)--(-2,0);
\draw[dashed] (-2.5,1)--(-2,1);
\draw[fill=gray!30,draw=none](-0.2,0.705) rectangle (1,0.795);
\draw[domain=-0.1:1,line width=3.1mm,gray!30]plot(\x, {0.1+ 0.15*cos(\x*180)});
\draw[domain=0:1,line width=0.2mm,dotted]plot(\x, {0.1+ 0.15*cos(\x*180)});
\draw[domain=0:1,line width=0.2mm,dotted]plot(\x, {0.75});
\draw[domain=0:1]plot(\x, { 0.065+0.1+0.15*cos(\x*180)});
\draw[domain=0:1]plot(\x, { -0.065+0.1+0.15*cos(\x*180)});
\draw (1,0.7)--(1,0.8);
\draw (1,-0.115)--(1,0.015);
\node at (-1.6,0.1){\small $\Pi_-$};
\node at (0.4,0.88){\small $L^{\eps}_-$};
\node at (0.4,0.32){\small $L^{\eps}_+$};
\node at (-0.15,0.75){\small $A_-$};
\node at (-0.15,0.25){\small $A_+$};
\node at (1.15,0.75){\small $\Sigma^{\eps}$};
\node at (1.15,-0.05){\small $\Sigma^{\eps}$};
\fill (0,0.25) circle (0.75pt);
\fill (0,0.75) circle (0.75pt);
\end{tikzpicture}
\caption{Geometry of the waveguide $\om^\eps$ with two truncated ligaments. \label{DomainOriginalHalf}} 
\end{figure}

\noindent Introduce the problem with Neumann Artificial Boundary Condition (ABC) at the end of the two truncated ligaments 
\begin{equation}\label{MainPbN}
\begin{array}{|rcll}
 \Delta u^\eps_N +\omega^2  u^\eps_N&=&0&\mbox{ in }\omega^\eps\\
 \partial_\nu u^\eps_N &=&0 &\mbox{ on }\partial\omega^\eps
\end{array}
\end{equation}
as well as the one with Dirichlet ABC 
\begin{equation}\label{MainPbD}
\begin{array}{|rcll}
 \Delta u^\eps_D +\omega^2  u^\eps_D&=&0&\mbox{ in }\omega^\eps\\
 \partial_\nu u^\eps_D &=&0 &\mbox{ on }\partial\omega^\eps\cap \partial\Omega^\eps\\[3pt]
 u^\eps_D &=&0 &\mbox{ on }\Sigma^{\eps}:=\partial\omega^\eps\setminus \partial\Omega^\eps.
\end{array}
\end{equation}
These problems admit the solutions
\begin{eqnarray}
u^\eps_{N1}(x,y)&=&\mrm{w}^+_1(x+1/2,y)+r^{\eps N}_{11}\mrm{w}^-_1(x+1/2,y)+r^{\eps N}_{12}\mrm{w}^-_2(x+1/2,y)+\dots \mbox{ in }\om^{\eps}\label{FieldN1}\\[3pt]
u^\eps_{N2}(x,y)&=&\mrm{w}^+_2(x+1/2,y)+r^{\eps N}_{21}\mrm{w}^-_1(x+1/2,y)+r^{\eps N}_{22}\mrm{w}^-_2(x+1/2,y)+\dots \mbox{ in }\om^{\eps}\label{FieldN2}
\end{eqnarray}
and
\begin{eqnarray}
u^\eps_{D1}(x,y)&=&\mrm{w}^+_1(x+1/2,y)+r^{\eps D}_{11}\mrm{w}^-_1(x+1/2,y)+r^{\eps D}_{12}\mrm{w}^-_2(x+1/2,y)+\dots \mbox{ in }\om^{\eps}\label{FieldD1}\\[4pt]
u^\eps_{D2}(x,y)&=&\mrm{w}^+_2(x+1/2,y)+r^{\eps D}_{21}\mrm{w}^-_1(x+1/2,y)+r^{\eps D}_{22}\mrm{w}^-_2(x+1/2,y)+\dots \mbox{ in }\om^{\eps}.\label{FieldD2}
\end{eqnarray}
Again, the ellipsis stand for remainders which decay at infinity with the rate $e^{-(4\pi^2-\om^2)^{1/2}|x|}$. With the reflection coefficients $r^{\eps N}_{ij}$, $r^{\eps D}_{ij}\in\Cplx$, we can form the scattering matrices of the two problems
\[
R^{\eps}_N:=\left(\begin{array}{cc}
r^{\eps N}_{11} & r^{\eps N}_{12} \\[6pt]
r^{\eps N}_{21} & r^{\eps N}_{22} 
\end{array}\right)\qquad\mbox{ and }\qquad 
R^{\eps}_D:=\left(\begin{array}{cc}
r^{\eps D}_{11} & r^{\eps D}_{12} \\[6pt]
r^{\eps D}_{21} & r^{\eps D}_{22} 
\end{array}\right).
\]
The scattering coefficients satisfy the relations of conservation of energy,
for $i=1,2$, 
\begin{equation}\label{RelConsNRJDemi}
|r^{\eps N}_{i1}|^2+|r^{\eps N}_{i2}|^2=1,\qquad\qquad |r^{\eps D}_{i1}|^2+|r^{\eps D}_{i2}|^2=1.
\end{equation}

\begin{lemma}\label{lemmaDecomposition}
We have the identities
\begin{equation}\label{RelationSymmetryIdentity}
R^{\eps}=\cfrac{R^{\eps}_N+R^{\eps}_D}{2}\qquad\mbox{ and }\qquad T^{\eps}=\cfrac{R^{\eps}_N-R^{\eps}_D}{2}\,.
\end{equation}
\end{lemma}
\begin{proof}
For $i=1,2$, define the function $e_i$ such that $e_i=u^{\eps}_i-(u^{\eps}_{Ni}+u^{\eps}_{Di})/2$ in $\om^\eps$ and $e_i(x,y)=u^{\eps}_i(x,y)-(u^{\eps}_{Ni}(-x,y)-u^{\eps}_{Di}(-x,y))/2$ in $\Om^\eps\setminus\overline{\om^{\eps}}$. Due to the boundary conditions satisfied by $u^{\eps}_{Ni}$ and $u^{\eps}_{Di}$, one can check that $e_i$ solves the problem (\ref{MainPb}). Then multiplying the equation $ \Delta e_i +\omega^2  e_i=0$ by $\overline{e_i}$, integrating by parts in $\Om^\eps_L:=\{(x,y)\in\Om^\eps\,|\,|x|<L\}$ and taking the imaginary part as $L\to+\infty$, we get, for $i=1,2$,
\[
\sum_{j=1}^2 \left|r^{\eps}_{ij}-\cfrac{r^{\eps N}_{ij}+r^{\eps D}_{ij}}{2}\right|^2+\sum_{j=1}^2 \left|t^{\eps}_{ij}-\cfrac{r^{\eps N}_{ij}-r^{\eps D}_{ij}}{2}\right|^2=0.
\]
This gives (\ref{RelationSymmetryIdentity}).
\end{proof}
\noindent From (\ref{RelationSymmetryIdentity}), we deduce that to obtain the relations (\ref{ConditionModalConverter}) characterising the mode conversion, we have to find geometries $\om^\eps$ where, when $\eps$ tends to zero,
\begin{equation}\label{RelationSymmetry}
R^{\eps}_N=\left(\begin{array}{cc}
0 & 1 \\[6pt]
1 & 0 
\end{array}\right)+o(1)\qquad\mbox{ and }\qquad R^{\eps}_D=\left(\begin{array}{cc}
0 & -1 \\[6pt]
-1 & 0 
\end{array}\right)+o(1).
\end{equation}
In order to get these particular values for $R^{\eps}_N$, $R^{\eps}_D$, as already mentioned, we will tuned precisely the parameters defining the thin ligaments. Let us describe in more details the geometry (see Figure \ref{DomainOriginalHalf}). Pick two different numbers $y_\pm\in(0;1)$ and set 
\[
A_{\pm}:=(-1/2,y_\pm).
\]
Let $\mathscr{L}_{\pm}^{\eps}$ be a simple 1D smooth curve connecting $A_{\pm}$ to some $B_{\pm}\in(Oy)$ (the position of the $B_{\pm}$ has no influence in the study). We assume that the curves $\mathscr{L}_{-}^{\eps}$ and $\mathscr{L}_{+}^{\eps}$ do not intersect, that the tangents to $\mathscr{L}_{\pm}^{\eps}$ at $A_{\pm}$, $B_{\pm}$ are parallel to the $(Ox)$ axis and that $\mathscr{L}_{\pm}^{\eps}$ is of length 
\begin{equation}\label{DefL}
\ell^\eps_{\pm}:=\ell_{\pm}+\eps \ell^\prime_{\pm}
\end{equation}
(similar to what has been done in \cite{NaCh21a,NaCh21b}). Here the values $\ell_{\pm}>0$, $\ell^\prime_\pm>0$ will be fixed later on to observe interesting phenomena. In a neighbourhood of $\mathscr{L}_{\pm}^{\eps}$, we introduce the local curvilinear coordinates $(n_{\pm},s_{\pm})$ where $s_{\pm}\in[0;\ell_{\pm}^{\eps}]$ is the arc length (with $s_{\pm}=0$ at $A_{\pm}$) and $n_{\pm}$ is the oriented distance to $\mathscr{L}_{\pm}^{\eps}$. Finally, we define the thin ligaments 
\[
L^{\eps}_{\pm}:=\{z\,|\,s_{\pm}\in[0;\ell_{\pm}^{\eps}),\,-\eps/2 <n_{\pm}<\eps/2\}.
\]
The initial ligaments $\mathcal{L}^{\eps}_{\pm}$ introduced in (\ref{DomainFirstDef}) are then defined by symmetrization of $L^{\eps}_{\pm}$ with respect to the $(Oy)$ axis. In the sequel, we will compute an asymptotic expansion of the functions $u^\eps_{Ni}$, $u^\eps_{Di}$ appearing in (\ref{FieldN1})--(\ref{FieldD2}) as $\eps$ tends to zero. This will give us an expansion of the matrices $R^{\eps}_N$, $R^{\eps}_D$. To proceed, first we introduce some auxiliary objects which will be useful in the analysis.

\section{Auxiliary objects} \label{SectionAux}

$\star$ Considering the limit $\eps\rightarrow0^+$ in the equation \eqref{MainPbN} restricted to the thin ligaments $L^{\eps}_{\pm}$, we are led to study the one-dimensional Helmholtz problem with mixed BC
\begin{equation}\label{Pb1D N}
\begin{array}{|ll}
\partial^2_{s}v+\omega^2v=0\qquad\mbox{ in }\mathscr{L}_{\pm}:=\{z\,|\,s_{\pm}\in(0;\ell_{\pm}),\,n_{\pm}=0\}\\[3pt]
v(0)=\partial_sv(\ell_{\pm})=0.
\end{array}
\end{equation}
Note that the condition $v(0)=0$ is imposed artificially. Eigenvalues and eigenfunctions (up to a multiplicative constant) of the problem (\ref{Pb1D N}) are given by
\[
(\pi (m+1/2)/\ell_{\pm})^2,\qquad\quad v(s_{\pm})=\sin(\pi(m+1/2) s_{\pm}/\ell_{\pm})\qquad\quad\mbox{ with }m\in\N:=\{0,1,2,3,\dots\}.
\]
$\star$ On the other hand, taking the limit $\eps\rightarrow0^+$ in the equation \eqref{MainPbD} restricted to $L^{\eps}_{\pm}$, we obtain the problem with Dirichlet BC
\begin{equation}\label{Pb1D D}
\begin{array}{|ll}
\partial^2_{s}v+\omega^2v=0\qquad\mbox{ in }\mathscr{L}_{\pm}\\[3pt]
v(0)= v(\ell_{\pm})=0.
\end{array}
\end{equation}
Eigenvalues and eigenfunctions (up to a multiplicative constant) of the problem \eqref{Pb1D D} are given by
\[
(\pi m/\ell_{\pm})^2,\qquad\quad v(s_{\pm})=\sin(\pi m s_{\pm}/\ell_{\pm})\qquad\quad\mbox{ with }m\in\N^{\ast}:=\{1,2,3,\dots\}.
\]
Importantly, in the sequel we shall choose $\ell_{-}$, $\ell_{+}$ such that at the limit $\eps\to0$, the ligament 
$L^{\eps}_{-}$ is resonant for the problem (\ref{Pb1D N}) while $L^{\eps}_{+}$ is resonant for the problem (\ref{Pb1D D}). In other words, we select $\ell_{-}$, $\ell_{+}$ such that 
\begin{equation}\label{ConditionResonance}
\om\,\ell_{-}=\pi(m_-+1/2) \qquad\mbox{ and }\qquad \om\,\ell_{+}=\pi m_+ 
\end{equation}
for some $m_-\in\N$, $m_+\in\N^{\ast}$. Let us emphasize that the limit problems (\ref{Pb1D N}), (\ref{Pb1D D}) are set in the fixed curves $\mathscr{L}_{\pm}$. But the true lengths $\ell^\eps_{\pm}=\ell_{\pm}+\eps \ell'_{\pm}$ of the ligaments $L^{\eps}_{\pm}$ (and not only their widths) depend on the parameter $\eps$. This is an essential element in the analysis to bring to light the resonant phenomena (see in the same spirit \cite{NaCh21a,NaCh21b}).\\

\begin{figure}[!ht]
\centering
\begin{tikzpicture}[scale=1.3,rotate=-90]
\begin{scope}
\clip (-1.8,2) -- (1.8,2) -- (1.8,0) -- (-1.8,0) -- cycle;
\draw[fill=gray!30,draw=none](0,2) circle (1.8);
\end{scope}
\draw[fill=gray!30,draw=none](-0.5,2) rectangle (0.5,3.5);
\draw (-1.8,2)--(-0.5,2)--(-0.5,3.5);
\draw (1.8,2)--(0.5,2)--(0.5,3.5);
\draw[dashed] (-0.5,3.5)--(-0.5,4);
\draw[dashed] (0.5,3.5)--(0.5,4);
\draw[dashed] (2.1,2)--(1.8,2);
\draw[dashed] (-2.1,2)--(-1.8,2);
\draw[thick,draw=gray!30](-0.495,2)--(0.495,2);
\node[mark size=1pt,color=black] at (0,2) {\pgfuseplotmark{*}};
\node[color=black] at (0,1.8) {\small $O$};
\begin{scope}[shift={(1.9,1)}]
\draw[->] (0,0)--(-0.5,0);
\draw[->] (0,0)--(0,0.5);
\node at (-0.66,0){\small$\xi_y$};
\node at (0,0.66){\small$\xi_x$};
\end{scope}
\node at (0,1){\small$\Xi^-$};
\node at (0,2.75){\small$\Xi^+$};
\node at (-1.5,2.2){\small$\Xi$}; 
\draw[dotted] (-0.5,2)--(0.5,2);
\end{tikzpicture}
\caption{Geometry of the inner field domain $\Xi$.\label{FrozenGeom}} 
\end{figure}
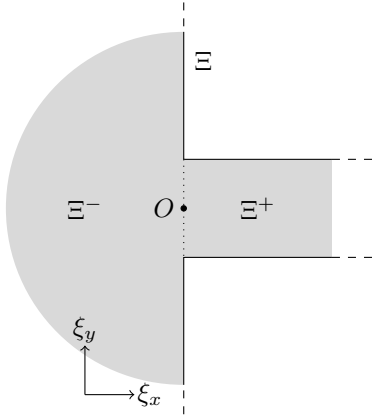

\noindent $\star$ Now we present a third problem which is involved in the construction of asymptotics and which will be used to describe the boundary layer phenomenon near the junction points $A_{\pm}$. To capture rapid variations of the field in the vicinity of $A_{\pm}$, we introduce the stretched coordinates $\xi^\pm=(\xi^\pm_x,\xi^\pm_y)=\eps^{-1}(z-A_\pm)=(\eps^{-1}(x+1/2),\,\eps^{-1}(y-y_{\pm}))$. Observing that 
\begin{equation}\label{Strechted1}
(\Delta_z+\om^2)u^{\eps}(\eps^{-1}(z-A_\pm))=\eps^{-2}\Delta_{\xi^{\pm}}u^{\eps}(\xi^{\pm})+\dots,
\end{equation}
we are led to consider the Neumann problem 
\begin{equation}\label{PbBoundaryLayer}
-\Delta_\xi Y=0\qquad\mbox{ in }\Xi,\quad\qquad \partial_\nu Y=0\quad\mbox{ on }\partial\Xi.
\end{equation}
Here $\Xi:=\Xi^-\cup\Xi^+\subset\mathbb{R}^2$ (see Figure \ref{FrozenGeom}) is the union of the half-plane $\Xi^-$ and the semi-strip $\Xi^+$ such that
\[
\Xi^-:={\mathbb R}^2_-=\{\xi=(\xi_x,\xi_y):\,\xi_x<0\},\qquad\qquad
\Xi^+:=\{\xi:\,\xi_x\geq 0, |\xi_y|<1/2\}.
\]
In the method of matched asymptotic expansions (see
the monographs \cite{VD,Ilin}, \cite[Chpt. 2]{MaNaPl} and others) that we will use, we will work with solutions of \eqref{PbBoundaryLayer} which are bounded or which have polynomial growth in the semi-strip as $\xi_x\rightarrow+\infty$ as well as logarithmic growth in the half plane as $|\xi|\rightarrow+\infty$. One of such solutions is evident and is given by $Y^0=1$.
Another solution, which is linearly independent with $Y^0$, is the unique function satisfying (\ref{PbBoundaryLayer}) and which has the representation
\begin{equation}\label{PolyGrowth}
Y^1(\xi)=\left\{
\begin{array}{ll}
\xi_x+C_\Xi+O(e^{-\pi \xi_x})& \mbox{ as }\xi_x\rightarrow+\infty,\quad \xi
\in \Xi^+\\[5pt]
\dsp\frac{1}{\pi}\ln \frac{1}{|\xi|}+O
 \Big(\frac{1}{|\xi|}\Big)& \mbox{ as }|\xi|\rightarrow+\infty,\quad \xi
\in \Xi^-.
\end{array}\right.
\end{equation}
Here, $C_\Xi$ is a universal constant whose value can be computed using conformal mapping, see for example \cite{Schn17}. Note that the coefficients in front of the growing terms in \eqref{PolyGrowth} are related due to the fact that a harmonic function has zero total flux at infinity. For the existence of $Y^1$ and the uniqueness of its definition, we refer the reader for example to \cite[Lemma 4.1]{BoCN18}.

\section{Asymptotic analysis for the problem with Neumann ABC}\label{SectionAsymptoNeumann}

In this section, we compute an asymptotic expansion of the fields $u^{\eps}_{N1}$, $u^{\eps}_{N2}$ appearing in (\ref{FieldN1}), (\ref{FieldN2}), as $\eps$ tends to zero. The final results are summarized in (\ref{AsymptoFinalResults1}), (\ref{AsymptoFinalResults2}). 

\subsection{Scattering of the first mode}

First, we focus our attention on the analysis for the field $u^{\eps}_{N1}$. To shorten notation, we remove the index ${}_{N1}$.\\
\newline
In the channel, we work with the ansatz
\begin{equation}\label{AnsatzWaveguides}
u^\eps=u^0+\eps u^\prime +\dots \quad\mbox{\rm in }\Pi_-,
\end{equation}
while in the thin ligaments, we consider the expansion 
\begin{equation}\label{AnsatzLigaments}
u^\eps(x,y)=\eps^{-1}v^{-1}_{\pm}(s_{\pm})+v^0_{\pm}(s_{\pm})+\dots\quad\mbox{\rm in }L^\eps_{\pm}.
\end{equation}
Here the ellipsis stand for higher order terms which are not important in our analysis. Taking the formal limit $\eps\to0^+$, we find that $v^{-1}_{\pm}$ must solve the homogeneous problem (\ref{Pb1D N}). Note in particular that the condition $v^{-1}_{\pm}(0)=0$ comes from the fact that the expansion (\ref{AnsatzWaveguides}) of $u^{\eps}$ in $\Pi_-$ remains bounded as $\eps$ tends to zero. Under the assumption (\ref{ConditionResonance}) for the lengths $\ell_{\pm}$, we must take $v^{-1}_{+}=0$ (we recall that $\mathscr{L}_+$ is not resonant for the problem (\ref{Pb1D N})) and $v^{-1}_{-}$ of the form 
\[
v^{-1}_{-}(s_-)=a {\bf v}(s_-)\qquad\mbox{ with }\qquad a\in\mathbb{C},\ {\bf v}(s)=\sin(\omega s).
\]  
Let us stress that the value of $a$ is unknown and will be fixed during the construction of the asymptotics of $u^\eps$. 
At $A_-$, the Taylor formula gives
\begin{equation}\label{DefInnerA}
\eps^{-1}v^{-1}_{-}(s_-)+v^0_{-}(s_-)=0+(C^{A_-} \xi^-_x  +v^0_{-}(0))+\dots\qquad\mbox{ with }\qquad C^{A_-}:=a\partial_s {\bf v}(0)=a\omega.
\end{equation}
Here $\xi^-_x=\eps^{-1}(x+1/2)$ is the stretched variable introduced just before (\ref{Strechted1}).\\
\newline
We look for an inner expansion of $u^{\eps}$ in the vicinity of $A_{-}$ of the form
\[
u^\eps(x)=C^{A_-}\,Y^1(\xi^-)+ c^{A_-}+\dots
\]
where $Y^1$ is introduced in (\ref{PolyGrowth}), $C^{A_-}$ is defined in (\ref{DefInnerA}) and $c^{A_-}$ is a constant to determine.\\
\newline 
Let us continue the matching procedure. Taking the limit $\eps\to0^+$, we find that the main term $u^0$ in (\ref{AnsatzWaveguides}) must solve the problem
\[
\Delta u^0 +\omega^2u^0=0\ \mbox{ in }\Pi_-,\qquad 
\partial_\nu u^0=0\mbox{ on }\partial\Pi_-\setminus \{A_-\},
\]
with the expansion 
\begin{equation}\label{DecompoZero}
u^0(x,y)=\mrm{w}^+_1(x+1/2,y)+r^0_1\,\mrm{w}^-_1(x+1/2,y)+r^0_2\,\mrm{w}^-_2(x+1/2,y)+\tilde{u}^0(x,y).
\end{equation}
Here $r^0_1,\,r^0_2\in\Cplx$ and $\tilde{u}^0$ decay exponentially at infinity. The coefficients $r^0_1,\,r^0_2$ will provide the first terms in the asymptotics of $r^{\eps N}_{11}$, $r^{\eps N}_{12}$, which by removing the indices $N$ and $1$, simply writes
\[
r^{\eps}_1=r^{0}_1+\dots\qquad\mbox{ and }\qquad r^{\eps}_2=r^{0}_2+\dots\,.
\]
Matching the behaviours of the inner and outer expansions of $u^{\eps}$ in $\Pi_-$, we find that at the point $A_{-}$, the function $u^0$ must expand as 
\[
u^0(x,y)= C^{A_-}\frac{1}{\pi}\ln \frac{1}{r^{A_-}}+U^0+O(r^{A_-})
\qquad \mbox{ as }r^{A_-}:=((x+1/2)^2+(y-y_-)^2)^{1/2}\rightarrow0^+,
\]
where $U^0$ is a constant. Observe that $u^0$ is singular at $A_{-}$. For $i=1,2$, define the function $W_i$ such that
\[
W_i(x,y)=\mrm{w}^+_i(x+1/2,y)+\mrm{w}^-_i(x+1/2,y).
\]
Note that we have $\Delta W_i +\omega^2W_i=0$ in $\Pi_-$ and $\partial_\nu W_i=0$ on $\partial\Pi_-$. Integrating by parts in 
\[
0=\int_{\Pi_-^\rho}(\Delta u^0 +\omega^2u^0)W_i-u^0\,(\Delta W_i +\omega^2 W_i)\,dz,
\]
with $\Pi_-^\rho:=\{(x,y)\in\Pi_-\,,x>-\rho\mbox{ and }r^A_{-}>1/\rho\}$, and taking the limit $\rho\to+\infty$, we get 
\begin{equation}\label{Part1System}
\begin{array}{|l}
2i\sqrt{\beta_1}(r^{0}_1-1)+2C^{A_-}=0\\[4pt]
2i\sqrt{\beta_2}r^{0}_2+2\cos(\pi y_-)\sqrt{2}C^{A_-}=0.
\end{array}
\end{equation}
From the expression of $C^{A_-}$ (see (\ref{DefInnerA})), this gives
\begin{equation}\label{equation1}
\begin{array}{|l}
r_{1}^{0}=1+ia\sqrt{\beta_1}\\[4pt]
r_{2}^{0}=ia\cos(\pi y_-)\sqrt{2}\beta_1/\sqrt{\beta_2}.
\end{array}
\end{equation}
Then matching the constant behaviour between the outer expansion and the inner  expansion inside $\Pi_-$, we get
\[
U^0=C^{A_-}\,\pi^{-1}\ln\eps+c^{A_-}=-C^{A_-}\,\pi^{-1}|\ln\eps|+c^{A_-}.
\]
This sets the value of $c^{A_-}$. However $U^0$ depends on $a$ and we have to explicit this dependence. For $u^0$, we have the decomposition
\begin{equation}\label{SecondDecompo}
u^0(x,y)=\mrm{w}^+_1(x+1/2,y)+\mrm{w}^-_1(x+1/2,y)+C^{A_-}\gamma_-
\end{equation}
where $\gamma_{-}$ are the outgoing functions such that
\begin{equation}\label{DefGamma}
\begin{array}{|rcll}
\Delta \gamma_{-}+\om^2\gamma_{-}&=&0&\mbox{ in }\Pi_-\\
\partial_\nu\gamma_{-}&=&\delta_{A_{-}}&\mbox{ on }\partial\Pi_-.
\end{array}
\end{equation}
Here $\delta_{A_{-}}$ stands for the Dirac delta function at $A_{-}$. Denote by $\Gamma_{-}$ the constant behaviour of $\gamma_{-}$ at $A_{-}$, that is the constant such that $\gamma_{-}$ behaves as 
\begin{equation}\label{DefGammaM}
\gamma_{-}(x,y)= \frac{1}{\pi}\ln \frac{1}{r^{A_-}}+\Gamma_{-}+O(r^{A_-})\qquad \mbox{ when }r^{A_-}\rightarrow0^+.
\end{equation}
Then from (\ref{SecondDecompo}), we derive
\[
U^0=\cfrac{2}{\sqrt{\beta_1}}+a\om \Gamma_-.
\]
Matching the constant behaviour at $A_{-}$ inside the thin ligament $L^{\eps}_-$, we obtain
\begin{equation}\label{BoundaryCondition1}
\begin{array}{lcl}
v_{-}^0(0)&=&C^{A_-}\,C_{\Xi}+c^{A_-} = U^0+C^{A_-}\,(\pi^{-1}|\ln\eps|+C_{\Xi})\\[5pt]
 &=& \cfrac{2}{\sqrt{\beta_1}}+a\om\,(\pi^{-1}|\ln\eps|+C_{\Xi}+\Gamma_-).
\end{array}
\end{equation}
Writing the compatibility condition so that the problem (\ref{Pb1D N}), supplemented with the condition (\ref{BoundaryCondition1}) instead of $v^0_-(0)=0$, admits a solution, we get
\begin{equation}\label{CompatibilityCondition}
v_{-}^0\partial_s\textbf{v}|_0-v_{-}^0\partial_s\textbf{v}|_{\ell_{-}}-(\textbf{v}\partial_s v_{-}^0|_0-\textbf{v}\partial_s v_{-}^0|_{\ell_{-}})=0.
\end{equation}
Since $\textbf{v}(0)=\partial_s\textbf{v}(\ell_{-})=0$, we obtain
\[
\om v_{-}^0(0)+(-1)^{m_{-}}\partial_sv_{-}^0(\ell_-)=0.
\]
On the other hand, from $\partial_\nu(\eps^{-1}a {\bf v})(\ell^\eps_-)+v^0_-(\ell^\eps_-)+\dots =0$, we infer that $\partial_s v^0_-(\ell_-)=\om^2 a\ell'_-\sin(\om \ell_-)=(-1)^{m_-}\om^2 a\ell'_-$. Thus we get
\begin{equation}\label{Part2System}
2/\sqrt{\beta_1}+a\om\,(\pi^{-1}|\ln\eps|+C_{\Xi}+\Gamma_-+\ell'_-)=0.
\end{equation}
Below, see Lemmas \ref{LemmaCReal} and \ref{lemmaRelConstants}, we prove that $C_{\Xi}\in\R$ and $\Im m\,(\om\Gamma_{-})=1+2\beta_1\cos(\pi y_{-})^2/\beta_2$. Thus we have
\[
a(\eta_-+i(1+2\beta_1\cos(\pi y_-)^2/\beta_2))=-2/\sqrt{\beta_1}
\]
with $\eta_-:=\om(\pi^{-1}|\ln\eps|+C_{\Xi}+\Re e\,\Gamma_-+\ell'_-)$. \\
\newline
Gathering (\ref{Part1System}) and (\ref{Part2System}), we obtain the system
\begin{equation}\label{system1}
\begin{array}{|l}
r_{1}^{0}=1+ia\sqrt{\beta_1}\\[4pt]
r_{2}^{0}=ia\cos(\pi y_-)\sqrt{2}\beta_1/\sqrt{\beta_2}\\[4pt]
a(\eta_-+i(1+2\beta_1\cos(\pi y_-)^2/\beta_2))=-2/\sqrt{\beta_1}.
\end{array}
\end{equation}
In particular, when we choose $\ell'_-$ such that $\eta_-=0$, that is for $\ell'_-=-(\pi^{-1}|\ln\eps|+C_{\Xi}+\Re e\,\Gamma_-)$, we have
\begin{equation}\label{DefParama}
a=\cfrac{2i}{\sqrt{\beta_1}\,(1+2\beta_1\cos(\pi y_-)^2/\beta_2)}
\end{equation}
and so 
\[
r_{1}^{0}=1-\cfrac{2}{1+2\beta_1\cos(\pi y_-)^2/\beta_2}=\cfrac{2\beta_1\cos(\pi y_-)^2/\beta_2-1}{2\beta_1\cos(\pi y_-)^2/\beta_2+1},\qquad\quad  r_{2}^{0}=-\cfrac{2\cos(\pi y_-)\sqrt{2\beta_1/\beta_2}}{2\beta_1\cos(\pi y_-)^2/\beta_2+1}\,.
\]
This ends the asymptotic analysis of $u^{\eps}_{N1}$, $r_{11}^{N\eps}$, $r_{12}^{N\eps}$ as $\eps$ tends to zero. Let summarize the results that we obtained. Assume that
\begin{equation}\label{DefRegLong}
\begin{array}{l}\ell_-^{\eps}= \pi(m_-+1/2)/\om-\eps\ell'_-\\
\phantom{\ell_-^{\eps}}=\pi(m_-+1/2)/\om-\eps(\pi^{-1}|\ln\eps|+C_{\Xi}+\Re e\,\Gamma_-)\end{array} \begin{array}{l}\qquad\mbox{ and }\qquad\ell_+^{\eps}= \ell_{+}+o(1),\\
\phantom{bidule}
\end{array}
\end{equation}
where $m_-\in\N$ and  $\ell_{+}>0$ is such that $\ell_{+}\notin\{\pi(m_-+1/2)/\om,\,m\in \N\}$. Then when $\eps$ tends to zero, we have the following expansions
\begin{equation}\label{AsymptoFinalResults1}
\fbox{$\begin{array}{l}
u^{\eps}_{N1}(x,y)=\mrm{w}^+_1(x+1/2,y)+\mrm{w}^-_1(x+1/2,y)+a\om\gamma_-(x,y)+o(1) \mbox{ in }\Pi_-,\\[6pt]
u^{\eps}_{N1}(x,y)=\eps^{-1}a\sin(\om s)+O(1) \mbox{ in }L^\eps_-,\quad u^{\eps}_{N1}(x,y)=O(1) \mbox{ in }L^\eps_+,\quad \mbox{ $a$ given by (\ref{DefParama}),}\\[6pt]
r_{11}^{N\eps}=\cfrac{2\beta_1\cos(\pi y_-)^2/\beta_2-1}{2\beta_1\cos(\pi y_-)^2/\beta_2+1}+o(1),\qquad r_{12}^{N\eps}=-\cfrac{2\cos(\pi y_-)\sqrt{2\beta_1/\beta_2}}{2\beta_1\cos(\pi y_-)^2/\beta_2+1}+o(1)\,.
\end{array}$}
\end{equation}
Here $\gamma_-$ is the function introduced in (\ref{DefGamma}). Let us formulate two comments on these results. First, we emphasize that the influence of the non resonant ligament $L^\eps_+$ appears only at the next order in $\eps$. This is an important point which allows us to decouple the two problems (\ref{MainPbN}) and (\ref{MainPbD}) (at least at first order in $\eps$). Second, we observe that the first terms $r_{11}^{N0}$, $r_{12}^{N0}$ in the asymptotic of 
$r_{11}^{N\eps}$, $r_{12}^{N\eps}$ are such that $|r_{11}^{N0}|^2+|r_{12}^{N0}|^2=1$. This is coherent with the identity of conservation of energy (\ref{RelConsNRJDemi}). 
\subsection{Scattering of the second mode}
The asymptotic analysis for $u^{\eps}_{N2}$ is very similar. We simply underline the main differences. First, (\ref{DecompoZero}) becomes
\begin{equation}\label{DecompoZeroBis}
u^0(x,y)=\mrm{w}^+_2(x+1/2,y)+r^0_1\,\mrm{w}^-_1(x+1/2,y)+r^0_2\,\mrm{w}^-_2(x+1/2,y)+\tilde{u}^0(x,y)
\end{equation}
so that (\ref{Part1System}) now writes 
\[
\begin{array}{|l}
2i\sqrt{\beta_1}r^{0}_1+2C^{A_-}=0\\[4pt]
2i\sqrt{\beta_2}(r^{0}_2-1)+2\cos(\pi y_-)\sqrt{2}C^{A_-}=0
\end{array}\qquad\Leftrightarrow\qquad\begin{array}{|l}
r_{1}^{0}=ia\sqrt{\beta_1}\\[4pt]
r_{2}^{0}=1+ia\cos(\pi y_-)\sqrt{2}\beta_1/\sqrt{\beta_2}.
\end{array}
\]
On the other hand, decomposition (\ref{SecondDecompo}) is now of the form
\begin{equation}\label{SecondDecompoBis}
u^0(x,y)=\mrm{w}^+_2(x+1/2,y)+\mrm{w}^-_2(x+1/2,y)+C^{A_-}\gamma_-,
\end{equation}
which implies 
\[
U^0=\cfrac{2\sqrt{2}\cos(\pi y_-)}{\sqrt{\beta_2}}+a\om \Gamma_-.
\]
Then (\ref{BoundaryCondition1}) becomes 
\[
v_{-}^0(0)= \cfrac{2\sqrt{2}\cos(\pi y_-)}{\sqrt{\beta_2}}+a\om\,(\cfrac{|\ln\eps|}{\pi}+C_{\Xi}+\Gamma_-).
\]
Writing the compatibility condition as in (\ref{CompatibilityCondition}), finally we arrive at the system
\[
\begin{array}{|l}
r_{1}^{0}=ia\sqrt{\beta_1}\\[4pt]
r_{2}^{0}=1+ia\cos(\pi y_-)\sqrt{2}\beta_1/\sqrt{\beta_2}\\[4pt]
a(\eta_-+i(1+2\beta_1\cos(\pi y_-)^2/\beta_2))=-2\sqrt{2}\cos(\pi y_-)/\sqrt{\beta_2}
\end{array}
\]
with again $\eta_-=\om(\pi^{-1}|\ln\eps|+C_{\Xi}+\Re e\,\Gamma_-+\ell'_-)$. For $\ell'_-$ chosen such that $\eta_-=0$, we find
\begin{equation}\label{DefParamaBis}
a=\cfrac{2\sqrt{2}i\cos(\pi y_-)}{\sqrt{\beta_2}\,(1+2\beta_1\cos(\pi y_-)^2/\beta_2)}
\end{equation}
and so 
\[
r_{1}^{0}=-\cfrac{2\cos(\pi y_-)\sqrt{2\beta_1/\beta_2}}{1+2\beta_1\cos(\pi y_-)^2/\beta_2},\qquad\quad r_{2}^{0}=1-\cfrac{4\beta_1\cos(\pi y_-)^2/\beta_2}{1+2\beta_1\cos(\pi y_-)^2/\beta_2}=\cfrac{1-2\beta_1\cos(\pi y_-)^2/\beta_2}{1+2\beta_1\cos(\pi y_-)^2/\beta_2}\,.
\]
Let us summarize the results for the asymptotic expansions of $u^{\eps}_{N2}$, $r_{21}^{N\eps}$, $r_{22}^{N\eps}$. Assume that $\ell_-^{\eps}$, $\ell_+^{\eps}$ are as in (\ref{DefRegLong}). Then when $\eps$ tends to zero, we have 
\begin{equation}\label{AsymptoFinalResults2}
\fbox{$\begin{array}{l}
u^{\eps}_{N2}(x,y)=\mrm{w}^+_2(x+1/2,y)+\mrm{w}^-_2(x+1/2,y)+a\om\gamma_-(x,y)+o(1) \mbox{ in }\Pi_-,\\[6pt]
u^{\eps}_{N2}(x,y)=\eps^{-1}a\sin(\om s)+O(1) \mbox{ in }L^\eps_-,\quad 
u^{\eps}_{N2}(x,y)=O(1) \mbox{ in }L^\eps_+,\quad \mbox{ $a$ given by (\ref{DefParamaBis}),}\\[6pt]
r_{21}^{N\eps}=-\cfrac{2\cos(\pi y_-)\sqrt{2\beta_1/\beta_2}}{1+2\beta_1\cos(\pi y_-)^2/\beta_2}+o(1),\qquad r_{22}^{N\eps}=\cfrac{1-2\beta_1\cos(\pi y_-)^2/\beta_2}{1+2\beta_1\cos(\pi y_-)^2/\beta_2}+o(1)\,.
\end{array}$}
\end{equation}
Here $\gamma_-$ is the function introduced in (\ref{DefGamma}). Again the influence of the non resonant ligament $L^\eps_+$ appears only at the next order in $\eps$ and the first terms $r_{21}^{N0}$, $r_{22}^{N0}$ in the asymptotic of 
$r_{21}^{N\eps}$, $r_{22}^{N\eps}$ are such that $|r_{21}^{N0}|^2+|r_{22}^{N0}|^2=1$. Observe also that for $y_-=1/2$, that is for a ligament $L^\eps_-$ located at the mid line of the strip $\Pi_-$, we have $a=0$ and the amplitude of the field does not blow up in $L^\eps_-$ as $\eps$ tends to zero. In this case, the resonance is not excited and $r_{21}^{N\eps}=o(1)$, $r_{22}^{N\eps}=1+o(1)$ which is what we expect.

\section{Asymptotic analysis for the problem with Dirichlet ABC}\label{SectionAsymptoDirichlet}

In this section, we turn our attention to the asymptotic analysis of the fields $u^{\eps}_{D1}$, $u^{\eps}_{D2}$ appearing in (\ref{FieldD1}), (\ref{FieldD2}), as $\eps$ tends to zero. The final results are summarized in (\ref{AsymptoFinalResults1D}), (\ref{AsymptoFinalResults2D}). \\
\newline
The approach is exactly the same as the one followed in the previous section. We simply underline the main differences. We start with the study for the field $u^{\eps}_{D1}$. We consider the same ansatz as in (\ref{AnsatzWaveguides}), (\ref{AnsatzLigaments}). Taking the formal limit $\eps\to0^+$, this time we find that $v^{-1}_{\pm}$ must solve the homogeneous problem (\ref{Pb1D D}). Under the assumption (\ref{ConditionResonance}) for the lengths $\ell_{\pm}$, we must take $v^{-1}_{-}=0$ (because $\mathscr{L}_-$ is not resonant for the problem (\ref{Pb1D D})) and $v^{-1}_{+}$ of the form 
\[
v^{-1}_{+}(s_+)=a {\bf v}(s_+)\qquad\mbox{ with }\qquad a\in\mathbb{C},\ {\bf v}(s)=\sin(\omega s).
\]  
Then the analysis is the same as previously with $A_-$ replaced by $A_+$. In particular, instead of working with the function $\gamma_-$ in (\ref{DefGamma}), we need to introduce $\gamma_{+}$ the outgoing function such that
\begin{equation}\label{DefGammaDirichlet}
\begin{array}{|rcll}
\Delta \gamma_{+}+\om^2\gamma_{+}&=&0&\mbox{ in }\Pi_-\\
\partial_\nu\gamma_{+}&=&\delta_{A_{+}}&\mbox{ on }\partial\Pi_-
\end{array}
\end{equation}
where $\delta_{A_{+}}$ stands for the Dirac delta function at $A_{+}$. Denote also $\Gamma_{+}$ the constant behaviour of $\gamma_{+}$ at $A_{+}$, that is the constant such that 
\begin{equation}\label{DefGammaP}
\gamma_{+}(x,y)= \frac{1}{\pi}\ln \frac{1}{r^{A_+}}+\Gamma_{+}+O(r^{A_+})\qquad \mbox{ when }r^{A_+}:=((x+1/2)^2+(y-y_+)^2)^{1/2}\rightarrow0^+.
\end{equation}
Now in the compatibility condition (\ref{CompatibilityCondition}), since $\textbf{v}(0)=\textbf{v}(\ell_{+})=0$, we obtain
\[
v_{+}^0(0)-(-1)^{m_{+}}v_{-}^0(\ell_+)=0.
\]
On the other hand, from $(\eps^{-1}a {\bf v}+v^0_++\dots)(\ell^\eps_+)=0$, we infer that $v^0_+(\ell_+)=-\om a\ell'_+\cos(\om \ell_+)=-(-1)^{m_+}\om  a\ell'_+$. Thus we get
\begin{equation}\label{Part2SystemDirichlet}
2/\sqrt{\beta_1}+a\om\,(\pi^{-1}|\ln\eps|+C_{\Xi}+\Gamma_++\ell'_+)=0,
\end{equation}
which leads to a system similar to (\ref{system1}). Assume that 
\begin{equation}\label{DefRegLongD}
\begin{array}{l}\ell_-^{\eps}= \ell_{-}+o(1),\qquad\mbox{ and }\\
\phantom{bidule}
\end{array}\qquad\ 
\begin{array}{l}\ell_+^{\eps}= \pi m_+/\om-\eps\ell'_+\\
\phantom{\ell_+^{\eps}}=\pi m_+/\om-\eps(\pi^{-1}|\ln\eps|+C_{\Xi}+\Re e\,\Gamma_+)\end{array} 
\end{equation}
where $\ell_{-}>0$ is such that $\ell_{-}\notin\{\pi m/\om,\,m\in \N^{\ast}\}$ and $m_+\in\N^{\ast}$. Then when $\eps$ tends to zero, we have 
\begin{equation}\label{AsymptoFinalResults1D}
\fbox{$\begin{array}{l}
u^{\eps}_{D1}(x,y)=\mrm{w}^+_1(x+1/2,y)+\mrm{w}^-_1(x+1/2,y)+a\om\gamma_+(x,y)+o(1) \mbox{ in }\Pi_-,\\[6pt]
u^{\eps}_{D1}(x,y)=O(1) \mbox{ in }L^\eps_-,\qquad u^{\eps}_{D1}(x,y)=\eps^{-1}a\sin(\om s)+O(1) \mbox{ in }L^\eps_+,\\[6pt]
\mbox{with }a=\cfrac{2i}{\sqrt{\beta_1}\,(1+2\beta_1\cos(\pi y_+)^2/\beta_2)},\mbox{ and  }\\[16pt]
r_{11}^{D\eps}=\cfrac{2\beta_1\cos(\pi y_+)^2/\beta_2-1}{2\beta_1\cos(\pi y_+)^2/\beta_2+1}+o(1),\qquad r_{12}^{D\eps}=-\cfrac{2\cos(\pi y_+)\sqrt{2\beta_1/\beta_2}}{2\beta_1\cos(\pi y_+)^2/\beta_2+1}+o(1).
\end{array}$}
\end{equation}
Working analogously for $u^{\eps}_{D2}$, for $\ell_-^{\eps}$, $\ell_+^{\eps}$ as in (\ref{DefRegLongD}), when $\eps$ tends to zero, we get
\begin{equation}\label{AsymptoFinalResults2D}
\fbox{$\begin{array}{l}
u^{\eps}_{D2}(x,y)=\mrm{w}^+_2(x+1/2,y)+\mrm{w}^-_2(x+1/2,y)+a\om\gamma_+(x,y)+o(1) \mbox{ in }\Pi_-,\\[6pt]
u^{\eps}_{D2}(x,y)=O(1) \mbox{ in }L^\eps_-,\qquad u^{\eps}_{D2}(x,y)=\eps^{-1}a\sin(\om s)+O(1) \mbox{ in }L^\eps_+,\\[6pt]
\mbox{with }a=\cfrac{2\sqrt{2}i\cos(\pi y_+)}{\sqrt{\beta_2}\,(1+2\beta_1\cos(\pi y_+)^2/\beta_2)},\mbox{ and  }\\[16pt]
r_{21}^{D\eps}=-\cfrac{2\cos(\pi y_+)\sqrt{2\beta_1/\beta_2}}{1+2\beta_1\cos(\pi y_+)^2/\beta_2}+o(1),\qquad r_{22}^{D\eps}=\cfrac{1-2\beta_1\cos(\pi y_+)^2/\beta_2}{1+2\beta_1\cos(\pi y_+)^2/\beta_2}+o(1).
\end{array}$}
\end{equation}
In (\ref{AsymptoFinalResults1D}), (\ref{AsymptoFinalResults2D}), the function $\gamma_+$ is the one introduced in (\ref{DefGammaDirichlet}). This time, it is the influence of the ligament $L^\eps_-$ which is negligible for these problems. Let us mention also that, as expected, we have the relations $|r_{i1}^{D0}|^2+|r_{i2}^{D0}|^2=1$ for the first terms $r_{i1}^{D0}$, $r_{i2}^{D0}$ in the asymptotic of $r_{i1}^{D\eps}$, $r_{i2}^{D\eps}$. 

\section{Analysis of the results}\label{AnalysisResults}

In this section, we gather the results of the previous steps to exhibit a waveguide acting as a mode converter. First, from the results (\ref{AsymptoFinalResults1}), (\ref{AsymptoFinalResults2}), we observe that by picking the parameter $y_-\in(0;1)$ such that
\begin{equation}\label{RelationPositionSlitN1}
\dsp 2\beta_1\cos(\pi y_-)^2/\beta_2-1=0 \qquad\Leftrightarrow\qquad  \cos(\pi y_-)^2=\cfrac{\beta_2}{2\beta_1}=\cfrac{\sqrt{\om^2-\pi^2}}{2\om}\,,
\end{equation}
when $\eps$ tends to zero, for $\ell_-^{\eps}$, $\ell_+^{\eps}$ as in (\ref{DefRegLong}), we have 
\[
R^{\eps}_N=\left(\begin{array}{cc}
0 & -\mrm{sign}(\cos(\pi y_-)) \\[6pt]
-\mrm{sign}(\cos(\pi y_-)) & 0 
\end{array}\right)+o(1).
\]
Note that for $\om\in(\pi;2\pi)$, the equations (\ref{RelationPositionSlitN1}) admit exactly two different solutions in $(0;1)$. And there is a unique $y_-\in(0;1)$ such that additionally there holds 
\begin{equation}\label{RelationPositionSlitN2}
-\mrm{sign}(\cos(\pi y_-))=1.
\end{equation}
In this case, we have the desired asymptotics (\ref{RelationSymmetry}) for $R^{\eps}_N$ as $\eps$ tends to zero.
\begin{remark}\label{Remark1}
The equation (\ref{RelationPositionSlitN1}) has been obtained by imposing $
r_{11}^{N0}=0$. It is quite fortunate that we can find a position of the ligament $L^{\eps}_-$ to achieve this. But it is even more fortunate that we can do it while imposing in the same time $r_{12}^{N0}=1$. Calculus we did that we do not present here indicate that this is very specific to the geometry $\Pi_-=(-\infty;1/2)\times(0;1)$ and does not hold in general for other domains. 
\end{remark}
\noindent Second, from the results (\ref{AsymptoFinalResults1D}), (\ref{AsymptoFinalResults2D}), we note that by picking the parameter $y_+\in(0;1)$ such that
\begin{equation}\label{RelationPositionSlitD1}
\dsp 2\beta_1\cos(\pi y_+)^2/\beta_2-1=0 \qquad\Leftrightarrow\qquad  \cos(\pi y_+)^2=\cfrac{\beta_2}{2\beta_1}=\cfrac{\sqrt{\om^2-\pi^2}}{2\om}\,,
\end{equation}
when $\eps$ tends to zero, for $\ell_-^{\eps}$, $\ell_+^{\eps}$ as in (\ref{DefRegLongD}), we have 
\[
R^{\eps}_N=\left(\begin{array}{cc}
0 & -\mrm{sign}(\cos(\pi y_+)) \\[6pt]
-\mrm{sign}(\cos(\pi y_+)) & 0 
\end{array}\right)+o(1).
\]
For $\om\in(\pi;2\pi)$, the equations (\ref{RelationPositionSlitD1}) admit exactly two different solutions in $(0;1)$. And there is a unique $y_+\in(0;1)$ such that additionally there holds 
\begin{equation}\label{RelationPositionSlitD2}
\mrm{sign}(\cos(\pi y_+))=1.
\end{equation}
Then with this choice, we obtain the desired asymptotics (\ref{RelationSymmetry}) for $R^{\eps}_D$ as $\eps$ tends to zero. Observe in particular that $y_-$ and $y_+$ are located symmetrically with respect to the line $y=1/2$. 
\begin{remark}\label{Remark2}
Again the possibility of imposing both $r_{11}^{D0}=0$ and $r_{12}^{D0}=-1$ is a small miracle which does not happen for generic $\Pi_-$.
\end{remark}
\noindent Finally from Lemma \ref{lemmaDecomposition}, we can state the following proposition, the main result of the article.
\begin{proposition}\label{MainPropo}
Assume that  
\begin{equation}\label{ParamTuned}
\begin{array}{l}
y_-\mbox{ solves }(\ref{RelationPositionSlitN1})-(\ref{RelationPositionSlitN2});\qquad\qquad \ell_-^{\eps}=\pi (m_-+1/2)/\om-\eps(\pi^{-1}|\ln\eps|+C_{\Xi}+\Re e\,\Gamma_-);\\[2pt]
y_+\mbox{ solves }(\ref{RelationPositionSlitD1})-(\ref{RelationPositionSlitD2});\qquad\qquad \ell_+^{\eps}=\pi m_+/\om-\eps(\pi^{-1}|\ln\eps|+C_{\Xi}+\Re e\,\Gamma_+).
\end{array}
\end{equation}
Then when $\eps$ tends to zero, the reflection and transmission matrices $R^\eps$, $T^{\eps}$ in (\ref{DefScaMat}) are such that

\begin{equation}\label{ConditionModalConverterMain}
 R^\eps  =\left(\begin{array}{cc}
0 & 0 \\[1pt]
0 & 0 
\end{array}\right)+o(1),\qquad\qquad T^\eps  =\left(\begin{array}{cc}
0 & 1 \\[1pt]
1 & 0 
\end{array}\right)+o(1).
\end{equation}
Here $m_-\in\N$, $m_+\in\N^{\ast}$ and $C_{\Xi}$, $\Gamma_-$, $\Gamma_+$ are respectively introduced in (\ref{PolyGrowth}), (\ref{DefGammaM}), (\ref{DefGammaP}).
\end{proposition}

\section{Numerics}\label{SectionNumerics}

In this section, we illustrate the results that we have obtained above. 
We take $\om=3\pi/2\in(\pi;2\pi)$. According to the case, we compute numerically the scattering solutions $u^{\eps}_{1}$, $u^{\eps}_{2}$, $u^{\eps}_{N1}$, $u^{\eps}_{N2}$, $u^{\eps}_{D1}$, $u^{\eps}_{D2}$ defined respectively in (\ref{Field1}), (\ref{Field2}), (\ref{FieldN1}), (\ref{FieldN2}), (\ref{FieldD1}), (\ref{FieldD2}). To proceed, we use a $\mrm{P2}$ finite element method in domains obtained by truncating either $\Om^{\eps}$ or $\om^{\eps}$. On the artificial boundary created by the truncation, a Dirichlet-to-Neumann operator with 15 terms serves as a transparent condition. For $\om=3\pi/2$, the critical lengths (\ref{ConditionResonance}) are given by
\begin{equation}\label{InterestingValuesLiga}
\ell_-=\cfrac{\pi}{\om}\,(m_-+1/2)=\cfrac{2}{3}\,(m_-+1/2)\qquad\mbox{ and }\qquad\ell_+=\cfrac{\pi}{\om}\, m_+=\cfrac{2}{3}\,m_+.
\end{equation}
Below we pick $m_-=1$, $m_+=2$ so that $\ell_-=1$, $\ell_+=4/3$. On the other hand, the parameters $y_{\pm}$ for the ordinates of the starting points of the ligaments are set by solving (\ref{RelationPositionSlitN1})--(\ref{RelationPositionSlitD2}) as indicated in Proposition \ref{MainPropo}. Once we have computed the $u^{\eps}_{i}$, $u^{\eps}_{Ni}$, $u^{\eps}_{Di}$, it is easy to obtain the scattering coefficients in the representations (\ref{Field1})--(\ref{Field2}), (\ref{FieldN1})--(\ref{FieldD2}). For example, for $R>1/2$, we have
\[
\begin{array}{ll}
r^{\eps N}_{11}=\dsp\int_{0}^1(u^{\eps}_{N1}(-R,y)-\mrm{w}^+_1(-R+1/2))\mrm{w}^+_1(-R+1/2)\,dy,\\[10pt]
r^{\eps N}_{12}=\dsp\int_{0}^1(u^{\eps}_{N1}(-R,y)-\mrm{w}^+_1(-R+1/2))\mrm{w}^+_2(-R+1/2)\,dy.
\end{array}
\]

\begin{figure}[!ht]
\centering
\includegraphics[width=0.8\textwidth]{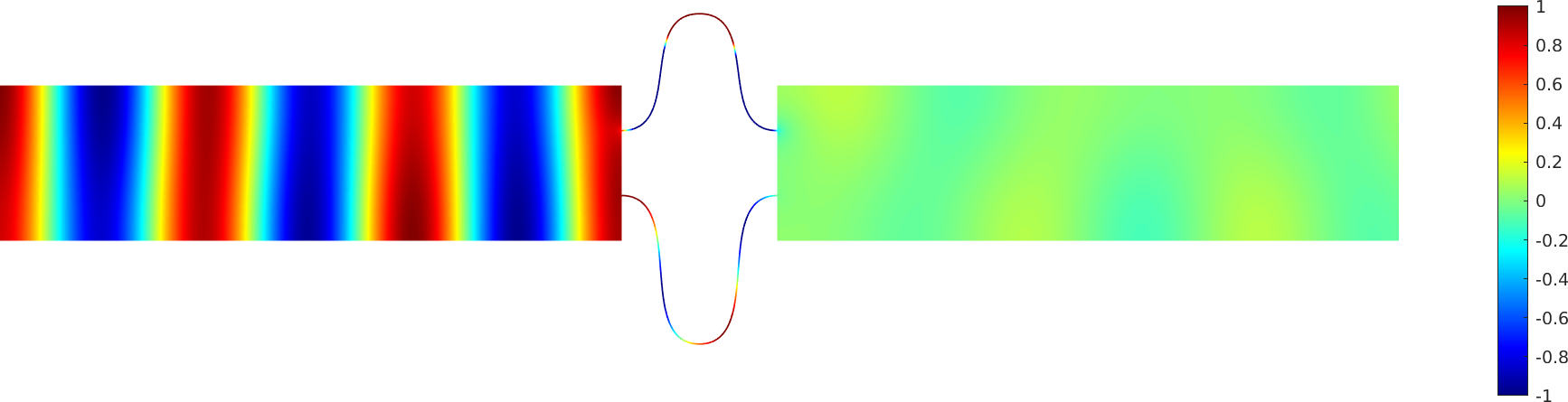}\\[10pt]
\includegraphics[width=0.8\textwidth]{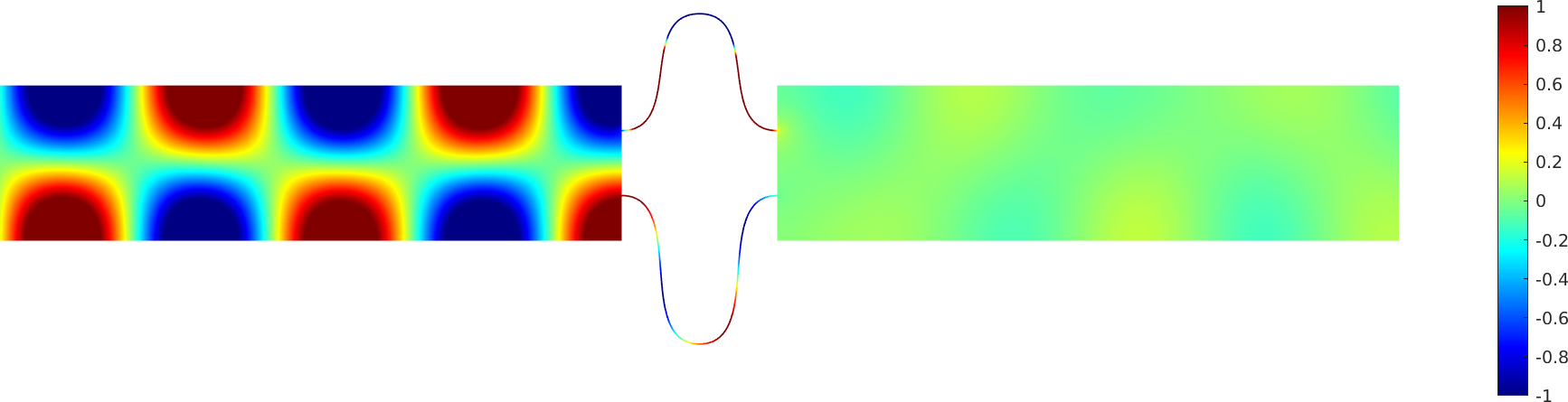}
\caption{Real parts of $u^\eps_1$ (top) and $u^\eps_2$ (bottom) for $\eps=0.01$. Here the lengths of the ligaments are close to the critical values (\ref{InterestingValuesLiga}) but not particularly selected to get mode conversion.}
\label{FieldNotTuned}
\end{figure}

In Figure \ref{FieldNotTuned}, we display the real parts of the fields $u^{\eps}_1$, $u^{\eps}_2$ in a geometry with some ligaments whose lengths are close to the critical values (\ref{InterestingValuesLiga}) but not particularly tuned. As expected, we observe that almost all the energy is backscatted and there is no mode conversion. More precisely, we find scattering matrices such that
\[
 R^\eps  \approx\left(\begin{array}{cc}
0.98-0.09i & 0.02+0.1i \\[1pt]
0.02+0.1i & 0.98-0.09i 
\end{array}\right),\qquad\qquad T^\eps \approx\left(\begin{array}{cc}
-0.02-0.12i & 0.02+0.07i \\[1pt]
0.02+0.07i & -0.02-0.12i 
\end{array}\right).
\]
This is coherent with (\ref{DefExpanNonCri}).~\\[-10pt]

\begin{figure}[!ht]
\centering
\includegraphics[width=0.78\textwidth]{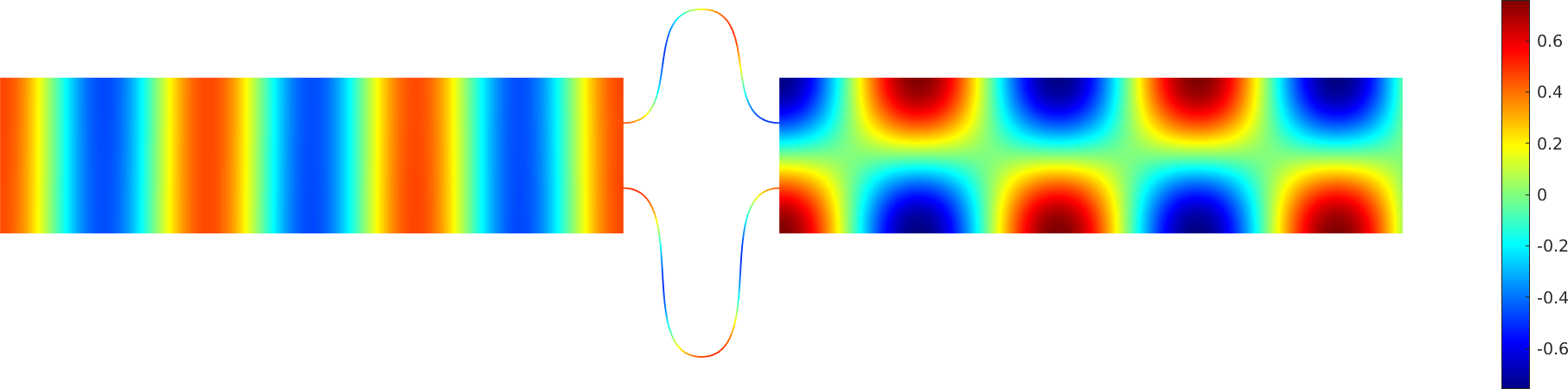}\\[10pt]
\includegraphics[width=0.78\textwidth]{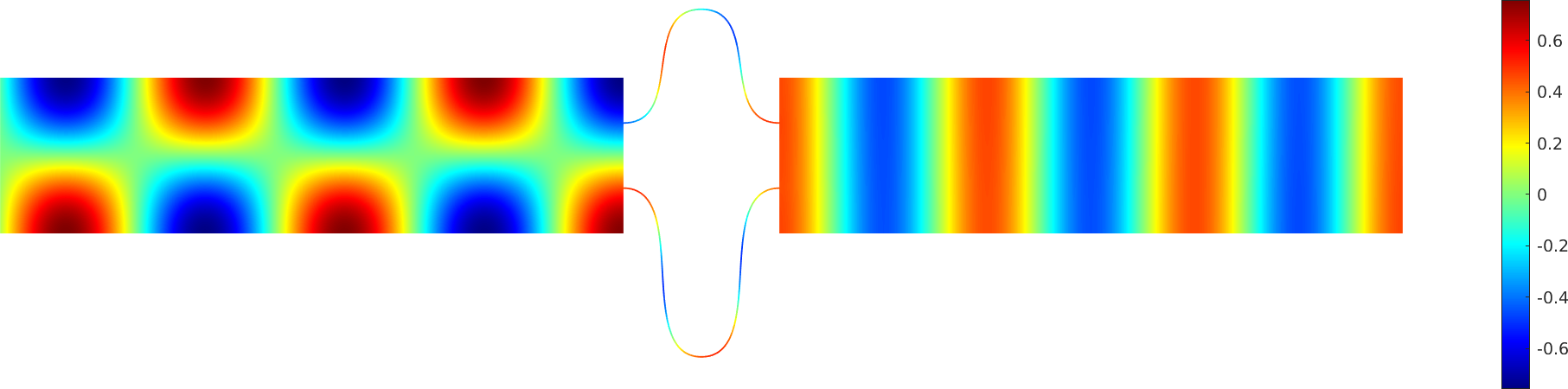}
\caption{Real parts of $u^\eps_1$ (top) and $u^\eps_2$ (bottom) for $\eps=0.01$. The lengths of the ligaments have been tuned to get mode conversion.}
\label{FieldFin}
\end{figure}

\begin{figure}[!ht]
\centering
\includegraphics[width=0.78\textwidth]{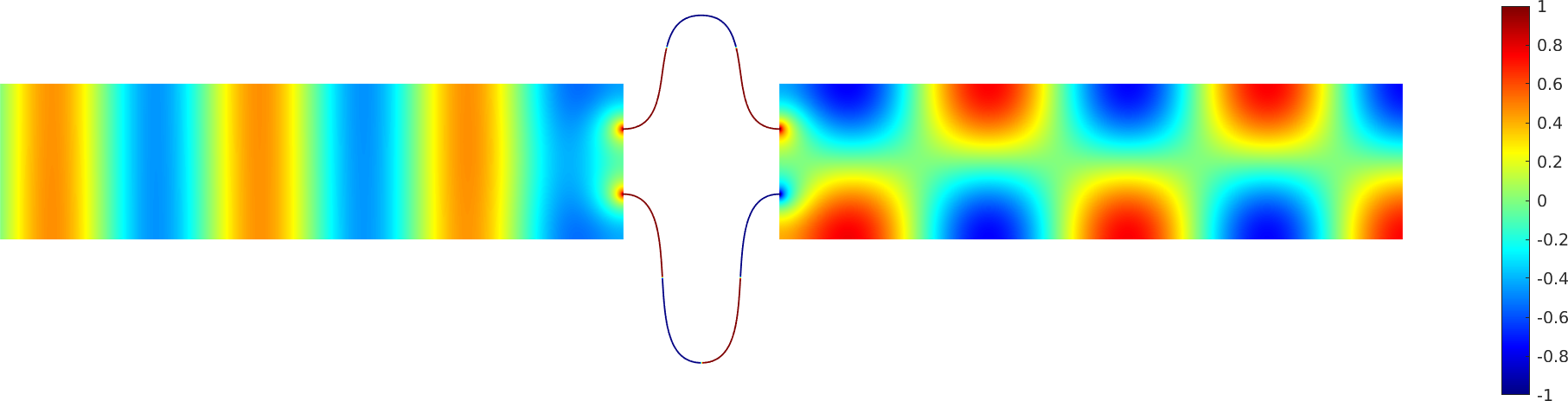}\\[10pt]
\includegraphics[width=0.78\textwidth]{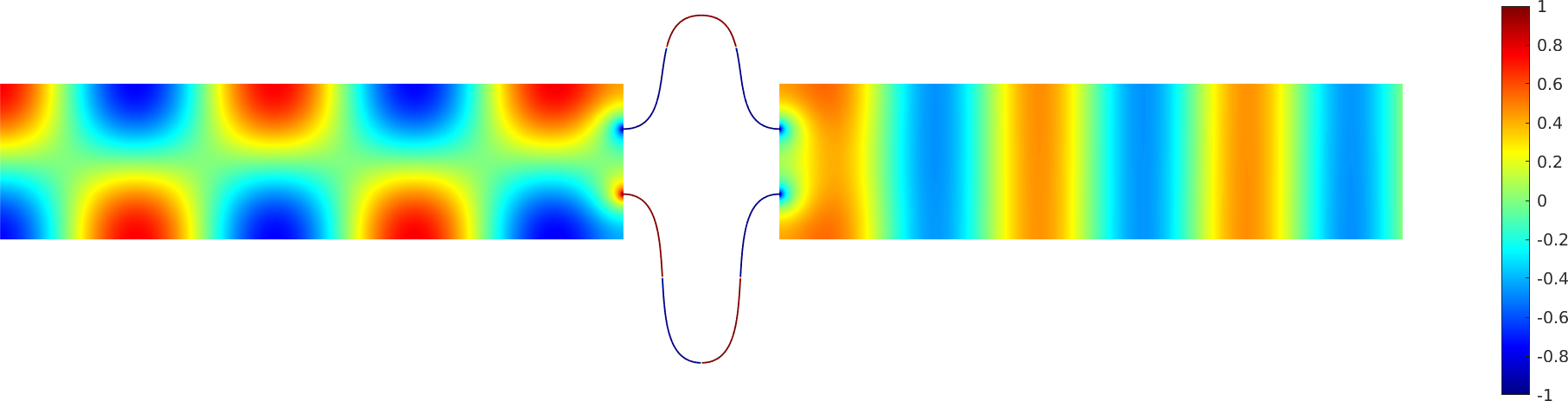}
\caption{Imaginary parts of $u^\eps_1$ (top) and $u^\eps_2$ (bottom) for $\eps=0.01$. The lengths of the ligaments have been tuned to get mode conversion.}
\label{FieldFinImag}
\end{figure}

\begin{figure}[!ht]
\centering
\includegraphics[width=0.38\textwidth]{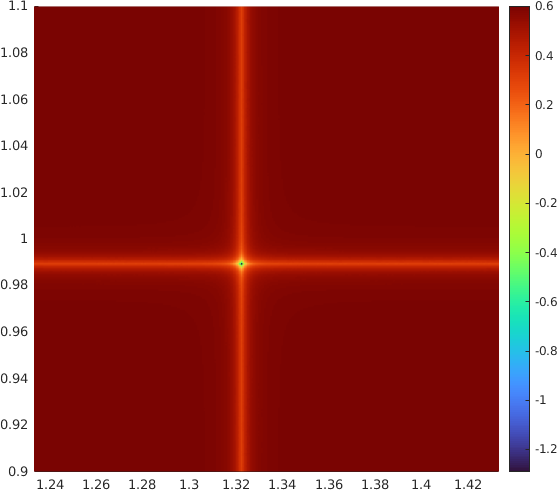}
\caption{Cost function $J(\ell^{\eps}_-,\ell^{\eps}_+)$  defined in (\ref{DefCostFunction}) for $\ell^{\eps}_-$, $\ell^{\eps}_+$ varying around the critical values (\ref{InterestingValuesLiga}). Here $\eps=0.01$.}
\label{CostFin}
\end{figure}

In Figures \ref{FieldFin}, \ref{FieldFinImag}, we display the fields $u^{\eps}_1$, $u^{\eps}_2$ in a geometry with some ligaments whose lengths have been carefully tuned to get mode conversion. We observe clearly the desired phenomenon. Numerically, we find
\[
R^\eps=\left(\begin{array}{cc}
0.3-0.4i & 0-8i \\[1pt]
0-8i & 0.2-0.4i
\end{array}\right)\,10^{-3},\qquad\qquad  T^\eps=\left(\begin{array}{cc}
0 & 1-10^{-3}i \\[1pt]
1-10^{-3}i &  0
\end{array}\right).
\]
This is coherent with the theory which predicts mode conversion up to an error in $O(\eps)$. On the other hand, formula (\ref{AsymptoFinalResults1}), (\ref{AsymptoFinalResults2}), (\ref{AsymptoFinalResults1D}), (\ref{AsymptoFinalResults2D}) indicate that in the ligaments, the real parts of the fields $u^\eps_{1}$, $u^\eps_{2}$ should be in $O(1)$ whereas the imaginary parts should be in $O(\eps^{-1})$ (note that each of the $a$ in (\ref{AsymptoFinalResults1}), (\ref{AsymptoFinalResults2}), (\ref{AsymptoFinalResults1D}), (\ref{AsymptoFinalResults2D}) are purely imaginary). Numerically this is indeed what we observe with an amplitude of the real part around $0.75$ while it goes to around $23$ for the imaginary part. In Figure \ref{CostFin}, we present the plot of the cost function $J$ such that
\begin{equation}\label{DefCostFunction}
J(\ell^{\eps}_-,\ell^{\eps}_+)= \ln\left(\,|R^\eps_N(\ell^{\eps}_-,\ell^{\eps}_+)-R^{\dagger}_N|+|R^\eps_D(\ell^{\eps}_-,\ell^{\eps}_+)-R^{\dagger}_D|\,\right)
\end{equation}
where $\ell^{\eps}_-$, $\ell^{\eps}_+$ vary around the values (\ref{InterestingValuesLiga}). In this definition, we take
\[
 R^{\dagger}_N:=\left(\begin{array}{cc}
1 & 0 \\[1pt]
0 & 1 
\end{array}\right),\qquad\qquad  R^{\dagger}_D:=\left(\begin{array}{cc}
-1 & 0 \\[1pt]
0 & -1 
\end{array}\right).
\]
We obtain a peak for some values $(\ell^{\eps}_-,\ell^{\eps}_+)=(\ell^{\star}_-,\ell^{\star}_+)$. We notice that there holds $\ell^{\star}_-<\ell_-=1$, $\ell^{\star}_+<\ell_+=4/3$. This is in agreement with the formula (\ref{ParamTuned}) of Proposition \ref{MainPropo}.

\begin{figure}[!ht]
\centering
\includegraphics[width=0.78\textwidth]{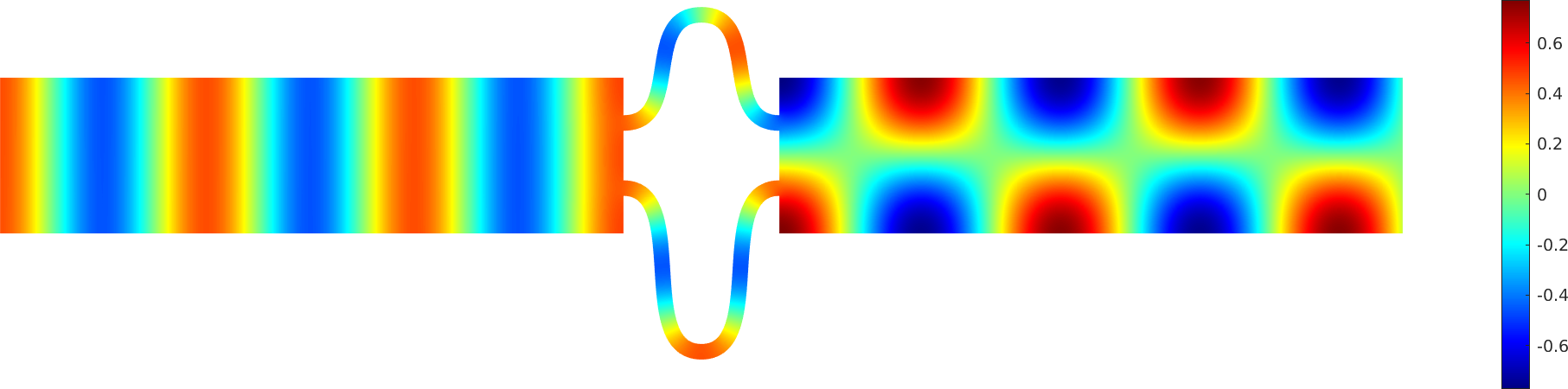}\\[10pt]
\includegraphics[width=0.78\textwidth]{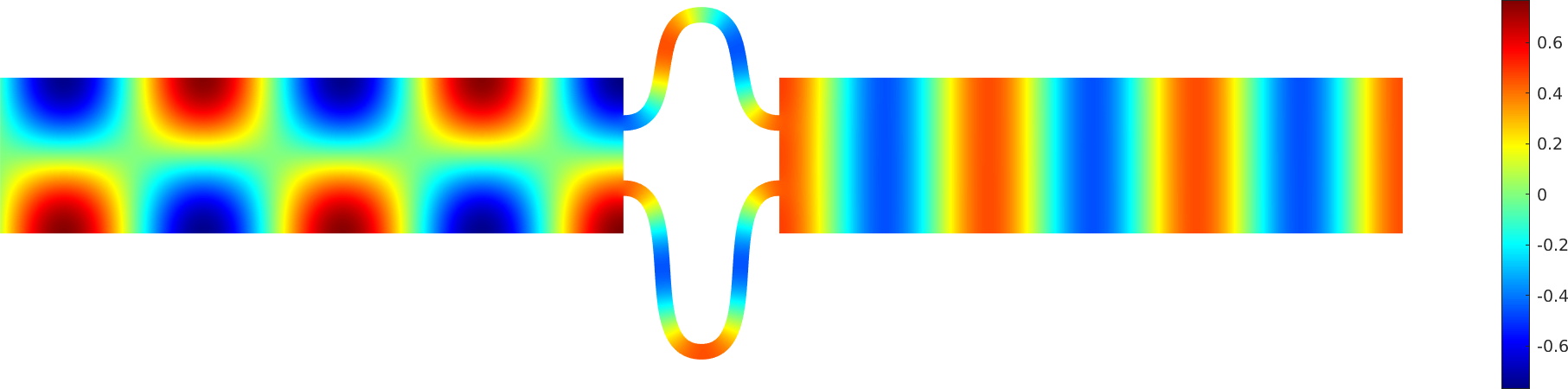}
\caption{Real parts of $u^\eps_1$ (top) and $u^\eps_2$ (bottom) for $\eps=0.1$. The lengths of the ligaments have been tuned to get mode conversion.}
\label{FieldGros}
\end{figure}

\begin{figure}[!ht]
\centering
\includegraphics[width=0.78\textwidth]{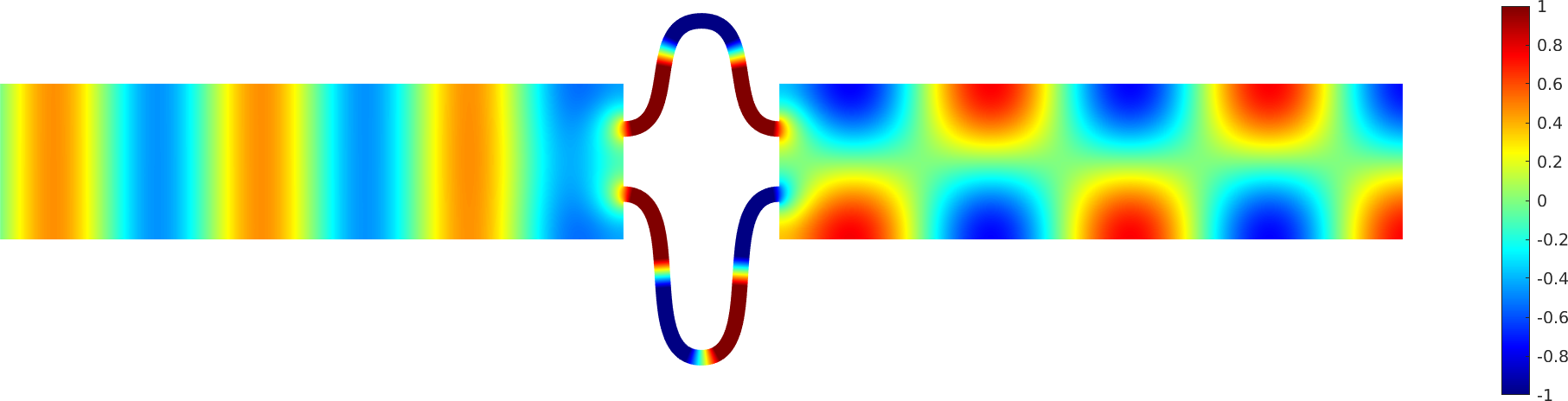}\\[10pt]
\includegraphics[width=0.78\textwidth]{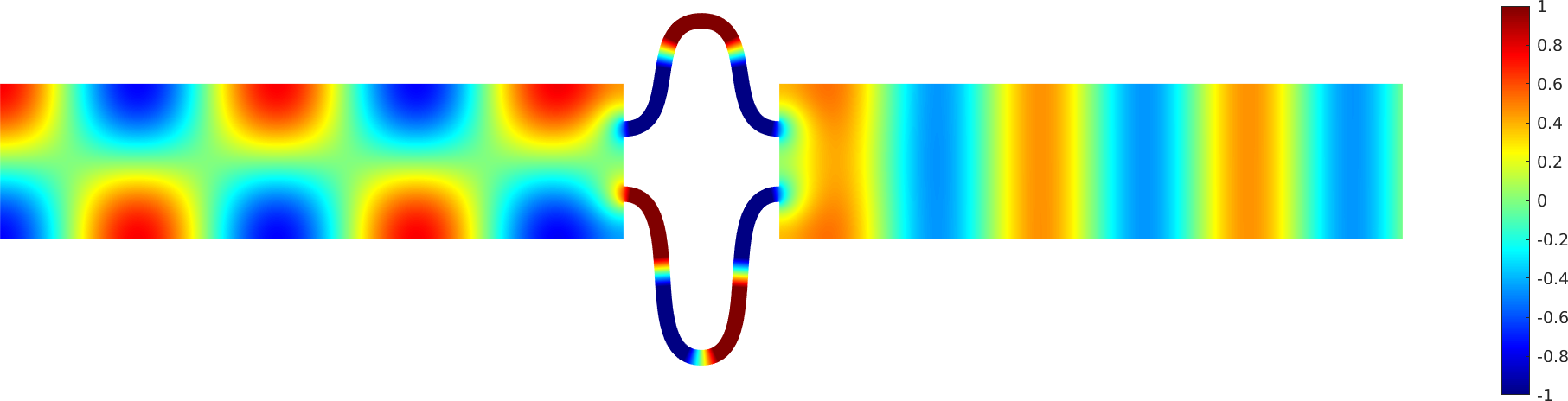}
\caption{Imaginary parts of $u^\eps_1$ (top) and $u^\eps_2$ (bottom) for $\eps=0.1$. The lengths of the ligaments have been tuned to get mode conversion.}
\label{FieldGrosImag}
\end{figure}

\begin{figure}[!ht]
\centering
\includegraphics[width=0.38\textwidth]{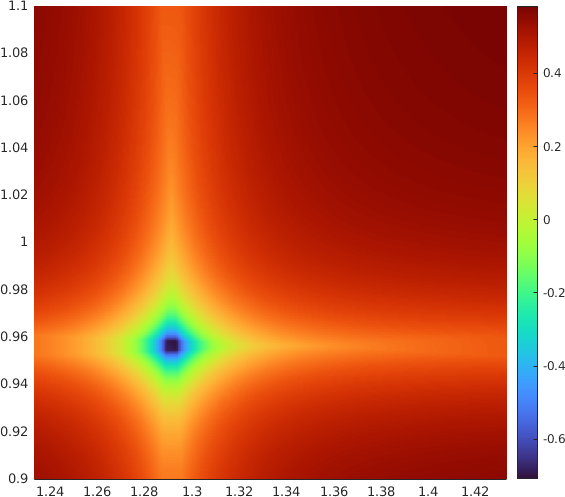}
\caption{Cost function $J(\ell^{\eps}_-,\ell^{\eps}_+)$  defined in (\ref{DefCostFunction}) for $\ell^{\eps}_-$, $\ell^{\eps}_+$ varying around  the critical values (\ref{InterestingValuesLiga}). Here $\eps=0.1$.}
\label{CostGros}
\end{figure}

In Figures \ref{FieldGros}--\ref{CostGros}, we display the same quantities as in Figures \ref{FieldFin}--\ref{CostFin} but with $\eps=0.1$ instead of $\eps=0.01$. Numerically for $(\ell^{\eps}_-,\ell^{\eps}_+)=(\ell^{\star}_-,\ell^{\star}_+)$, we find
\[
R^\eps=\left(\begin{array}{cc}
-3-4.7i & 0-0.1i \\[1pt]
0-0.1i & -0.1-4.7i
\end{array}\right)\,10^{-2},\qquad  T^\eps=\left(\begin{array}{cc}
(-0.1+1.4i)\,10^{-3} & 0.997-0.05i \\[1pt]
0.997-0.05i  &  (-0.5+1.4i)\,10^{-3}
\end{array}\right).
\]
Interestingly, we note that the mode conversion phenomenon still appear reasonably with ligaments which are not that thin. Of course the smaller $\eps$ is, the better the transmission and the conversion are. But by  comparing Figures \ref{CostFin} and \ref{CostGros}, we remark that the variation of the scattering coefficients gets even quicker as $\eps$ is small. This is a fact which can be inferred from the asymptotic analysis. It indicates that the mode conversion is more robust to perturbations of the setting for rather thick ligaments. In other words, when $\eps$ is very small, the lengths of the ligaments have to be tuned very precisely to observe the mode conversion.

\newpage

\section{Concluding remarks}\label{sectionConclu}

\begin{figure}[!ht]
\centering
\includegraphics[width=0.8\textwidth]{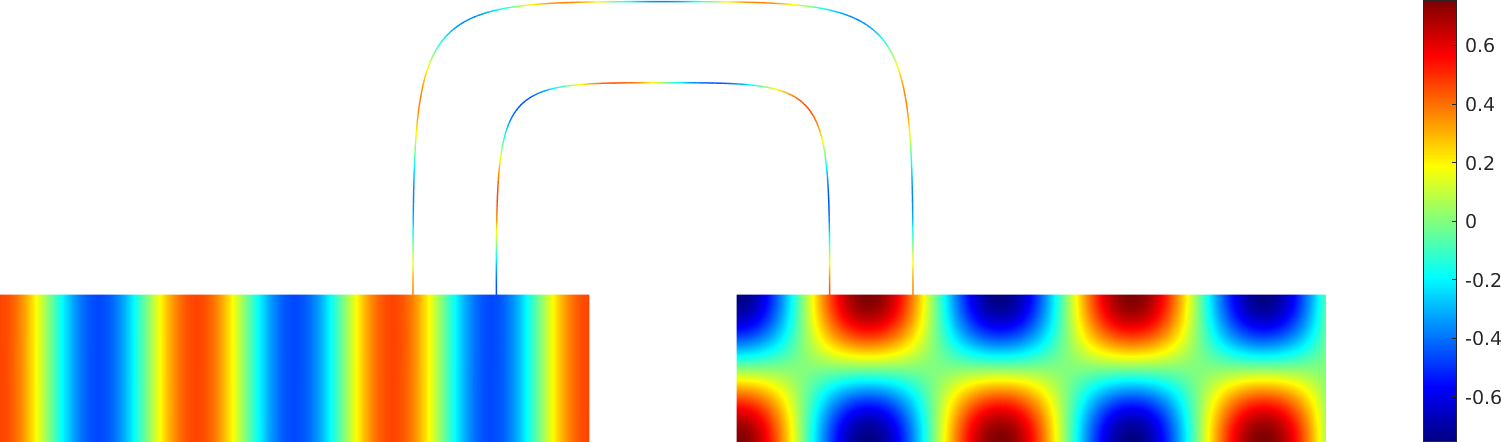}\\[10pt]
\includegraphics[width=0.8\textwidth]{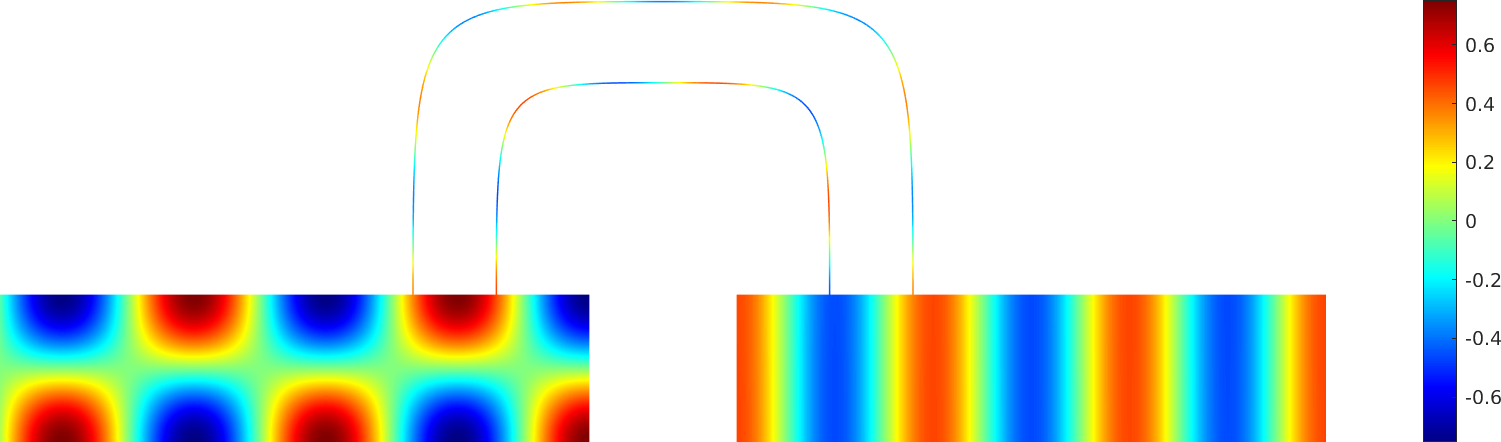}
\caption{Real parts of $u^\eps_1$ (top) and $u^\eps_2$ (bottom). The lengths of the ligaments have been tuned to get mode conversion.}
\label{FieldVertical}
\end{figure}

\noindent \textit{i)} We can place the thin ligaments on other parts of the waveguide, for example as depicted in Figure \ref{FieldVertical}. In this case, the results of the asymptotic analysis are a bit different, in particular the abscissa of the starting point of the ligaments play a role, but the method is completely similar. Note that this offers more degrees of freedom which can be useful if one wishes to work at higher wave number with more than two propagating modes. Let us underline that in this configuration, several targets can be desired for the transmission matrix.\\

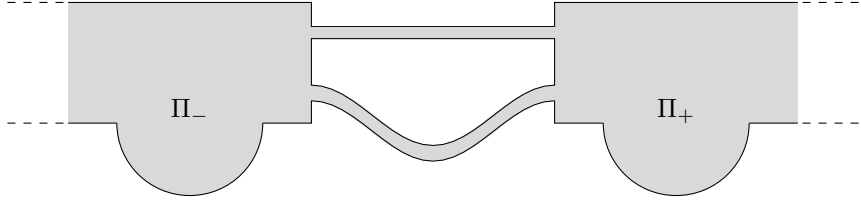
\begin{figure}[!ht]
\centering
\begin{tikzpicture}[scale=1.6]
\draw[fill=gray!30](-0.2,0.7) rectangle (2,0.8);
\draw[fill=gray!30,draw=none](-2,0) rectangle (0,1);
\draw[fill=gray!30,draw=none](2,0) rectangle (4,1);
\draw (-2,0)--(0,0)--(0,1)--(-2,1);
\draw (4,0)--(2,0)--(2,1)--(4,1);
\draw[fill=gray!30,draw=none] (-1,0) circle (0.6) ;
\draw[fill=gray!30,draw=none] (3,0) circle (0.6) ;
\draw (-1.6,0) arc (180:360:0.6);
\draw (2.4,0) arc (180:360:0.6);
\draw[dashed] (-2.5,0)--(-2,0);
\draw[dashed] (-2.5,1)--(-2,1);
\draw[dashed] (4.5,0)--(4,0);
\draw[dashed] (4.5,1)--(4,1);
\draw[fill=gray!30,draw=none](-0.2,0.705) rectangle (2.1,0.795);
\draw[domain=-0.1:2.1,line width=2mm,gray!30]plot(\x, { 0.25*cos(\x*180)}  );
\draw[domain=0:2]plot(\x, { 0.065+0.25*cos(\x*180)}  );
\draw[domain=0:2]plot(\x, { -0.065+0.25*cos(\x*180)}  );
\node at (-1,0.1){\small $\Pi_-$};
\node at (3,0.1){\small $\Pi_+$};
\end{tikzpicture}
\caption{Geometry of a waveguide with symmetric bumps. \label{Bumps}} 
\end{figure}

\noindent \textit{ii)} We considered a quite academic geometry with two straight channels. One may wonder if this assumption could be relaxed and if for example one could work in domains as illustrated in Figure \ref{Bumps} with symmetric bumps. The answer is no in general. The decomposition with the Neumann and Dirichlet problems would be the same and the asymptotic procedure would be very similar. This is an interest of the method proposed here compared for example to the technique of \cite{DeGr18,DeGr20} based on decomposition in Fourier series or the one of \cite{LiZh18,LiSZ19} relying on integral equations with an explicit kernel, we do not need separation of variables in $\Pi_{\pm}$. But as already announced in Remarks \ref{Remark1}, \ref{Remark2}, for a generic domain, we can not position the ligaments to get (\ref{RelationSymmetry}). In such a situation, it is necessary to work with a higher number of ligaments. But then coupling effects between the ligaments appear and they are not so simple to deal with.\\
\newline
\textit{iii)} On the other hand, we worked in 2D but what we proposed could be adapted in 3D. However the asymptotic procedure would be different. This comes in particular from the fact that in 3D, the $Y^1$ appearing in (\ref{PolyGrowth}) would have a different expansion at infinity (see more details in \cite{NaCh21a}). \\
\newline
\textit{iv)} The trick consisting in working a geometry which is symmetric with respect to the vertical axis is quite important in the analysis. It allows us to uncouple the influence of the two ligaments. Without it, we can still proceed to the asymptotic analysis in a similar way. But the results we obtain are less simple to exploit to get the mode conversion  because then coupling constants depending non explicitly on the positions of the ligaments appear in the expansions. \\
\newline
\textit{v)} What we proposed here is very specific to the Neumann BC and cannot be adapted for the Dirichlet BC (quantum waveguides). The reason is that with Dirichlet BC almost nothing passes through the thin ligaments. Another idea has to be found to design a mode converter for a waveguide with Dirichlet BC.

\section*{Appendix: auxiliary results}\label{Appendix}

\begin{lemma}\label{LemmaCReal}
The constant $C_{\Xi}$ appearing in the decomposition (\ref{PolyGrowth}) of the function $Y^1$ is real.
\end{lemma}
\begin{proof}
Since there holds $\Delta Y^1=0$ in $\Xi$, for all $\rho>0$, we have 
\[
0=\int_{\Xi_\rho} (Y^1-\overline{Y^1})\Delta Y^1-\Delta(Y^1-\overline{Y^1})Y^1\,d\xi_xd\xi_y
\]
with $\Xi_\rho:=\{(\xi_x,\xi_y)\in\Xi,\,\xi_x<0\mbox{ and }|\xi|<\rho\}\cup \{(\xi_x,\xi_y)\in[0;\rho)\times(-1/2;1/2)\}$. Integrating by parts and taking the limit $\rho\to+\infty$, we get $C_\Xi-\overline{C_\Xi}=0$. This shows that $C_\Xi$ is real.
\end{proof}

\begin{lemma}\label{lemmaRelConstants}
The constants $\Gamma_{\pm}$ corresponding to the constant behaviour of $\gamma_{\pm}$ at $A_{\pm}$ (see (\ref{DefGammaM}), (\ref{DefGammaP})) are such that 
\[
\Im m\,(\om\Gamma_{\pm})=1+2\beta_1\cos(\pi y_{\pm})^2/\beta_2. 
\]
\end{lemma}
\begin{proof}
Since the functions $\gamma_{\pm}$ are outgoing, we have the expansions
\[
\gamma_{\pm}(x,y)=s_{\pm1}\mrm{w}^-_1(x+1/2,y)+s_{\pm2}\mrm{w}^-_2(x+1/2,y)+\tilde{\gamma}_{\pm}(x,y)
\]
where $s_{\pm i}\in\Cplx$ and $\tilde{\gamma}_{\pm}$ are exponentially decaying at infinity. For $i=1,2$, integrating by parts in
\[
0=\int_{\Pi_-^\rho}(\Delta \gamma_{\pm} +\omega^2\gamma_{\pm})W_i-\gamma_{\pm}\,(\Delta W_i +\omega^2W_i)\,dz,
\]
and taking the limit $\rho\to+\infty$, we obtain 
\begin{equation}\label{relationCoefSca}
s_{\pm1}=i/\sqrt{\beta_1}\qquad\mbox{ and }\qquad s_{\pm2}=i\cos(\pi y_{\pm})\sqrt{2}/\sqrt{\beta_2}.
\end{equation}
On the other hand, integrating by parts in
\[
0=\int_{\Pi_-^\rho}(\Delta \gamma_{\pm} +\omega^2\gamma_{\pm})\overline{\gamma_{\pm}}-\gamma_{\pm}\,(\Delta \overline{\gamma_{\pm}} +\omega^2\overline{\gamma_{\pm}})\,dz,
\]
and taking again the limit $\rho\to+\infty$, we obtain $2(|s_{\pm1}|^2+|s_{\pm2}|^2)-2\Im m\,\Gamma=0$. From (\ref{relationCoefSca}), this yields the desired result.
\end{proof}

\section*{Acknowledgments} 
The research of S.A. Nazarov was supported by the grant No. 17-11-01003 of the Russian Science Foundation.

\bibliography{Bibli}

\def\cprime{$'$}
\begin{thebibliography}{10}

\bibitem{BaNa15}
{F.L.} Bakharev and {S.A.} Nazarov.
\newblock Gaps in the spectrum of a waveguide composed of domains with
  different limiting dimensions.
\newblock {\em Sib. Math. J.}, 56(4):575--592, 2015.

\bibitem{Beal73}
{J.T.} Beale.
\newblock Scattering frequencies of resonators.
\newblock {\em Comm. Pure Appl. Math.}, 26(4):549--563, 1973.

\bibitem{BoCN18}
{A.-S.} {Bonnet-Ben Dhia}, {L.} Chesnel, and {S.A.} Nazarov.
\newblock Perfect transmission invisibility for waveguides with sound hard
  walls.
\newblock {\em J. Math. Pures Appl.}, 111:79--105, 2018.

\bibitem{BoTr10}
{\'E}.~Bonnetier and F.~Triki.
\newblock {Asymptotic of the Green function for the diffraction by a perfectly
  conducting plane perturbed by a sub-wavelength rectangular cavity}.
\newblock {\em Math. Method. Appl. Sci.}, 33(6):772--798, 2010.

\bibitem{BrHS20}
R.~Brand{\~a}o, {J.R.} Holley, and O.~Schnitzer.
\newblock Boundary-layer effects on electromagnetic and acoustic extraordinary
  transmission through narrow slits.
\newblock {\em arXiv preprint arXiv:2006.04276}, 2020.

\bibitem{BrSc20}
R.~Brand{\~a}o and O.~Schnitzer.
\newblock Asymptotic modeling of helmholtz resonators including thermoviscous
  effects.
\newblock {\em Wave Motion}, page 102583, 2020.

\bibitem{CXJD06}
P.~Cheben, {D.X.} Xu, S.~Janz, and A.~Densmore.
\newblock Subwavelength waveguide grating for mode conversion and light
  coupling in integrated optics.
\newblock {\em Opt. Express}, 14(11):4695--4702, 2006.

\bibitem{ChNa20Distrib}
L.~Chesnel and {S.A.} Nazarov.
\newblock Design of an acoustic energy distributor using thin resonant slits.
\newblock {\em arXiv preprint arXiv:2007.13216v1}, 2020.

\bibitem{DeGr18}
A.~Delitsyn and {D.S.} Grebenkov.
\newblock Mode matching methods for spectral and scattering problems.
\newblock {\em The Quarterly Journal of Mechanics and Applied Mathematics},
  71(4):537--580, 2018.

\bibitem{DeGr20}
A.~Delitsyn and {D.S.} Grebenkov.
\newblock Resonance scattering in a waveguide with identical thick barriers.
\newblock {\em arXiv preprint arXiv:2001.07060}, 2020.

\bibitem{Gady93}
{R.R.} Gadyl'shin.
\newblock {Characteristic frequencies of bodies with thin spikes. I.
  Convergence and estimates}.
\newblock {\em Math. Notes}, 54(6):1192--1199, 1993.

\bibitem{Gady05}
{R.R.} Gadyl'shin.
\newblock On the eigenvalues of a ``dumbbell with a thin handle''.
\newblock {\em Izv. Math.}, 69(2):265--329, 2005.

\bibitem{Micro1}
{J.W.} Guo, L.~Sun, {X.J.} Niu, {X.Z.} Zhang, W.~Lu, {W.H.} Zhang, {Y.C.} Feng,
  and {H.W.} Zhao.
\newblock 24 ghz microwave mode converter optimized for superconducting ecr ion
  source secral.
\newblock {\em Rev. Sci. Instrum.}, 87(2):02A708, 2016.

\bibitem{HoSc19}
{J.R.} Holley and O.~Schnitzer.
\newblock Extraordinary transmission through a narrow slit.
\newblock {\em Wave Motion}, 91:102381, 2019.

\bibitem{HoHu05}
{B.M.} Holmes and {D.C.} Hutchings.
\newblock Realization of novel low-loss monolithically integrated passive
  waveguide mode converters.
\newblock {\em IEEE Photonics Technol. Lett.}, 18(1):43--45, 2005.

\bibitem{Ilin}
A.~M. Il'in.
\newblock {\em Matching of asymptotic expansions of solutions of boundary value
  problems}, volume 102 of {\em Transl. Math. Monogr.}
\newblock AMS, Providence, RI, 1992.

\bibitem{JoTo06}
P.~Joly and S.~Tordeux.
\newblock {Matching of asymptotic expansions for wave propagation in media with
  thin slots I: The asymptotic expansion}.
\newblock {\em SIAM Multiscale Model. Simul.}, 5(1):304--336, 2006.

\bibitem{KTSM09}
{S.-H.} Kim, R.~Takei, Y.~Shoji, and T.~Mizumoto.
\newblock Single-trench waveguide {TE-TM} mode converter.
\newblock {\em Opt. Express}, 17(14):11267--11273, 2009.

\bibitem{KoMM94}
{V.A.} Kozlov, {V.G.} Maz'ya, and {A.B.} Movchan.
\newblock Asymptotic analysis of a mixed boundary value problem in a
  multi-structure.
\newblock {\em Asymptot. Anal.}, 8(2):105--143, 1994.

\bibitem{Krieg}
G.A. Kriegsmann.
\newblock Complete transmission through a two-dimensional difffraction grating.
\newblock {\em SIAM J. Appl. Math.}, 65(1):24--42, 2004.

\bibitem{KuTW88}
H.~Kumri{\'c}, M.~Thumm, and R.~Wilhelm.
\newblock {Optimization of mode converters for generating the fundamental TE01
  mode from TE06 gyrotron output at 140 GHz}.
\newblock {\em Int. J. Electron}, 64(1):77--94, 1988.

\bibitem{LAHI00}
W.~Lawson, {M.R.} Arjona, {B.P.} Hogan, and {R.L.} Ives.
\newblock The design of serpentine-mode converters for high-power microwave
  applications.
\newblock {\em IEEE Trans. Microw. Theory Tech.}, 48(5):809--814, 2000.

\bibitem{Lebb19}
N.~Lebbe.
\newblock {\em Contribution in topological optimization and application to
  nanophotonics}.
\newblock PhD thesis, Universit{\'e} Grenoble Alpes, 2019.

\bibitem{LDOHG19}
N.~Lebbe, C.~Dapogny, E.~Oudet, K.~Hassan, and A.~Gliere.
\newblock Robust shape and topology optimization of nanophotonic devices using
  the level set method.
\newblock {\em J. Comput. Phys.}, 395(0):710--746, 2019.

\bibitem{LGHDO19}
N.~Lebbe, A.~Gli{\`e}re, K.~Hassan, C.~Dapogny, and E.~Oudet.
\newblock Shape optimization for the design of passive mid-infrared photonic
  components.
\newblock {\em Opt. Quant. Electron.}, 51(5):166, 2019.

\bibitem{LiSZ19}
J.~Lin, S.~Shipman, and H.~Zhang.
\newblock A mathematical theory for {Fano} resonance in a periodic array of
  narrow slits.
\newblock {\em arXiv preprint arXiv:1904.11019}, 2019.

\bibitem{LiZh17}
J.~Lin and H.~Zhang.
\newblock Scattering and field enhancement of a perfect conducting narrow slit.
\newblock {\em SIAM J. Appl. Math.}, 77(3):951--976, 2017.

\bibitem{LiZh18}
J.~Lin and H.~Zhang.
\newblock {Scattering by a periodic array of subwavelength slits I: field
  enhancement in the diffraction regime}.
\newblock {\em Multiscale Model. Sim.}, 16(2):922--953, 2018.

\bibitem{LiMF12}
V.~Liu, {D. A.B.} Miller, and S.~Fan.
\newblock Ultra-compact photonic crystal waveguide spatial mode converter and
  its connection to the optical diode effect.
\newblock {\em Optics express}, 20(27):28388--28397, 2012.

\bibitem{LuKG98}
E.~Lun\'eville, {J.-M.} Krieg, and E.~Giguet.
\newblock An original approach to mode converter optimum design.
\newblock {\em IEEE Trans. Microw. Theory Tech.}, 46(1):1--9, 1998.

\bibitem{MaNaPl}
{V.G.} {Maz'ya}, {S.A.} Nazarov, and {B.A.} Plamenevski{\u\i}.
\newblock {\em {Asymptotic theory of elliptic boundary value problems in
  singularly perturbed domains, Vol. 1}}.
\newblock {Birkh\"{a}user}, Basel, 2000.
\newblock Translated from the original German 1991 edition.

\bibitem{Naza96}
{S.A.} Nazarov.
\newblock Junctions of singularly degenerating domains with different limit
  dimensions 1.
\newblock {\em J. Math. Sci. (N.Y.)}, 80(5):1989--2034, 1996.

\bibitem{Naza05}
{S.A.} Nazarov.
\newblock Asymptotic analysis and modeling of the jointing of a massive body
  with thin rods.
\newblock {\em J. Math. Sci. (N.Y.)}, 127(5):2192--2262, 2005.

\bibitem{NaCh21a}
{S.A.} Nazarov and L.~Chesnel.
\newblock Abnormal transmission of waves through a thin canal connecting two
  acoustic waveguides.
\newblock {\em Dokl. Ross. Akad. Nauk. Fizika, Tekhn. nauki.}, 496:22--27,
  2021.
\newblock English transl.: Doklady Physics. 2021. V. 66, to appear.

\bibitem{NaCh21b}
{S.A.} Nazarov and L.~Chesnel.
\newblock Anomalies of propagation of acoustic waves in two semi-infinite
  cylinders connected by a thin flattened canal.
\newblock {\em Zh. Vychisl. Mat. i Mat. Fiz.}, 61:135--152, 2021.
\newblock English transl.: Comput. Math. and Math. Physics. 2021. V. 61, to
  appear.

\bibitem{OhLe14}
D.~Ohana and U.~Levy.
\newblock Mode conversion based on dielectric metamaterial in silicon.
\newblock {\em Optics express}, 22(22):27617--27631, 2014.

\bibitem{Micro2}
S.~Peng, C.~Yuan, H.~Zhong, and Y.~Fan.
\newblock Design and experiment of a cross-shaped mode converter for high-power
  microwave applications.
\newblock {\em Rev. Sci. Instrum.}, 84(12):124703, 2013.

\bibitem{Schn17}
O.~Schnitzer.
\newblock Spoof surface plasmons guided by narrow grooves.
\newblock {\em Phys. Rev. B}, 96(8):085424, 2017.

\bibitem{VD}
M.~Van~Dyke.
\newblock {\em Perturbation methods in fluid mechanics}.
\newblock The Parabolic Press, Stanford, Calif., 1964.

\end{thebibliography}
\bibliographystyle{plain}
\end{document}